\documentclass[11pt]{article}
\usepackage[letterpaper,margin=1in]{geometry}
\usepackage[T1]{fontenc}
\usepackage{lmodern,microtype,amsmath,amssymb,amsthm,mathtools}
\usepackage{algorithm,algpseudocode,booktabs,array,enumitem}
\usepackage[numbers,sort&compress]{natbib}
\usepackage{bm}
\usepackage{placeins}
\usepackage[colorlinks=true,linkcolor=blue,citecolor=blue,urlcolor=blue]{hyperref}
\hypersetup{pdftitle={Envelope Cubic Regularized Newton Method for Nonconvex Strongly-Concave Minimax Optimization},pdfauthor={}}
\setlist[itemize]{leftmargin=1.4em,itemsep=3pt,topsep=5pt}
\allowdisplaybreaks[2]

\newtheorem{theorem}{Theorem}
\newtheorem{lemma}[theorem]{Lemma}
\newtheorem{proposition}[theorem]{Proposition}
\newtheorem{observation}[theorem]{Observation}

\theoremstyle{definition}
\newtheorem{assumption}{Assumption}
\newtheorem{definition}{Definition}
\theoremstyle{remark}

\usepackage{thm-restate}

\def\sN{{\mathbb{N}}}
\def\sO{{\mathbb{O}}}

\newcommand{\R}{\mathbb R}

\newcommand{\eps}{\varepsilon}
\newcommand{\norm}[1]{\lVert #1\rVert}
\newcommand{\ip}[2]{\langle #1,#2\rangle}
\newcommand{\cO}{\mathcal O}
\newcommand{\tO}{\widetilde{\mathcal O}}
\newcommand{\argmin}{\operatorname*{arg\,min}}
\newcommand{\argmax}{\operatorname*{arg\,max}}
\newcommand{\supp}{\operatorname{supp}}
\newcommand{\Lip}{\operatorname{Lip}}

\newcommand{\bx}{{\bf x}}
\newcommand{\by}{{\bf y}}
\newcommand{\bz}{{\bf z}}
\newcommand{\bxi}{{\boldsymbol\xi}}
\newcommand{\bu}{{\bf u}}
\newcommand{\bv}{{\bf v}}
\newcommand{\bw}{{\bf w}}
\newcommand{\bs}{{\bf s}}
\newcommand{\bg}{{\bf g}}
\newcommand{\ba}{{\bf a}}

\newcommand{\bA}{{\bf A}}
\newcommand{\bB}{{\bf B}}

\newcommand{\bH}{{\bf H}}
\newcommand{\bI}{{\bf I}}

\newcommand{\bM}{{\bf M}}

\newcommand{\bP}{{\bf P}}

\newcommand{\bU}{{\bf U}}

\newcommand{\bzero}{{\bf 0}}

\def\vx{{\bm{x}}}
\def\vy{{\bm{y}}}

\title{Envelope Cubic Regularized Newton Method for\\Nonconvex-Strongly-Concave Minimax Optimization}
\author{Qihao Zhou\thanks{Equal Contribution.}\qquad Lesi Chen\footnotemark[1]\qquad Luo Luo\qquad Chengchang Liu}
\date{}
\begin{document}
\maketitle
\vspace{-2.4em}
\begin{abstract}
We study nonconvex-strongly-concave minimax problem $\min_{\mathbf{x}\in\mathbb{R}^{d_x}}\max_{\mathbf{y}\in\mathbb{R}^{d_y}}f(\mathbf{x},\mathbf{y})$ in which the objective has $\rho$-Lipschitz Hessian and condition number $\kappa$.
We propose envelope cubic regularized Newton method (ECRN) which applies a refined inexact cubic regularized Newton method to the envelope of the primal function $\Phi(\cdot) : = \max_{\mathbf{y}\in\mathbb{R}^{d_y}}f(\cdot,\mathbf{y})$. 
We show that ECRN finds approximate second-order stationary points of $\Phi(\cdot)$ with $\widetilde{\mathcal{O}}(\sqrt\rho\,\kappa^{3/4}\varepsilon^{-3/2})$ second-order oracles, improving the condition number dependence $\mathcal{O}(\kappa^{2})$ achieved by the best known second-order methods.
Our improvement comes from the key observation that the Hessian Lipschitz parameter of 
the envelope is $\mathcal{O}(\rho\kappa^{3/2})$,
which is sharper than the order of $\mathcal{O}(\rho\kappa^{3})$ for the function~$\Phi(\cdot)$.  
We also conduct a lower bound on second-order methods to show that our complexity is optimal up to a logarithmic factor.
\end{abstract}

\section{Introduction}\label{sec:intro}
We study the nonconvex--strongly-concave minimax problem
\begin{equation}\label{eq:problem}
\min_{\bx\in\R^{d_x}}\max_{\by\in\R^{d_y}}f(\bx,\by),
\end{equation}
where $f(\vx,\vy)$ has an $L$-Lipschitz gradient, is possibly nonconvex in $\bx$, and is $\mu$-strongly concave in $\by$. We write $\kappa:=L/\mu$ for the condition number and seek approximate stationary points of the primal function $\Phi(\bx):=\max_{\by\in\R^{d_y}}f(\bx,\by)$. 
Such problems arise in game theory~\cite{von1947theory}, adversarial learning~\cite{zhang2018mitigating,liu2022partial}, robust optimization~\cite{ben2009robust}, AUC maximization~\cite{hanley1982meaning,ying2016stochastic,yuan2021large}, and reinforcement learning~\cite{wai2018multi,paternain2022safe}. 

For the first-order methods, gradient descent ascent requires $\cO(\kappa^2\eps^{-2})$ gradient evaluations to find an $\eps$-approximate first-order stationary point of $\Phi(\cdot)$ \citep{lin2020gda}. \citet{lin2020near} improved the condition-number dependence through an inexact accelerated proximal-point method. 
Its central construction is the primal Moreau envelope:
\begin{equation}\label{eq:envelope}
\Phi_\gamma(\bz):=\min_{\bu\in\R^{d_x}}\Bigl\{\Phi(\bu)+\frac{\gamma}{2}\norm{\bu-\bz}^2\Bigr\}
=\min_{\bu\in\R^{d_x}}\max_{\by\in\R^{d_y}}\Bigl\{f(\bu,\by)+\frac{\gamma}{2}\norm{\bu-\bz}^2\Bigr\}.
\end{equation}
The regularized objective $f(\bu,\by)+(\gamma/2)\norm{\bu-\bz}^2$ is $L$-strongly convex in $\bu$ and $\mu$-strongly concave in $\by$ by taking $\gamma = 2L$.
An accelerated first-order solver therefore requires $\tO(\sqrt{L^2/(L\mu)})=\tO(\sqrt\kappa)$ evaluations per subproblem. 
Combining these inner solves with $\cO(L\eps^{-2})$ outer proximal steps gives $\tO(L\sqrt\kappa\,\eps^{-2})$ first-order evaluations \citep{lin2020near}, matching the lower bound of \citet{zhang2021complexity} up to logarithmic factors.
When the Hessian of $f(\cdot,\cdot)$ is $\rho$-Lipschitz continuous, the Hessian of the primal function $\Phi(\cdot)$ is also Lipschitz continuous \citep{luo2022,chen2023cubic} with parameter of $\mathcal{O}(\rho\kappa^3)$. 
Under this additional smoothness, \citet{wang2024} propose an inexact accelerated proximal-point method which finds an $\eps$-approximate first-order stationary point of $\Phi(\cdot)$ with $\tO(L^{1/2}\rho^{1/4}\kappa^{5/4}\eps^{-7/4})$ function and gradient.\footnote{Their Corollary~2.20 gives $\tO(L^{1/2}L_2^{1/4}\sqrt\kappa\,\eps^{-7/4})$ under the assumption that the Hessian of $\Phi(\cdot)$ is $L_2$-Lipschitz continuous. Replacing $L_2=\cO(\rho\kappa^3)$ leads to the complexity of $\tilde{\mathcal{O}}(L^{1/2}\rho^{1/4}\kappa^{5/4}\epsilon^{-7/4})$.}

We target to find the $(\epsilon,\mathcal{O}(\sqrt{\epsilon}))$-approximate second-order stationary points (SOSP) of $\Phi(\cdot)$.
It is natural to apply inexact cubic regularization to $\Phi(\cdot)$ which  takes $\cO(\sqrt{\rho\kappa^{3}}\eps^{-3/2})$ outer iterations~\cite{luo2022,chen2023cubic}.
The dual maximizer can be approximated by accelerated gradient ascent at each iteration, which adds a factor $\tO(\sqrt\kappa)$ to the gradient cost and leads to the total $\tO(\sqrt\rho\,\kappa^2\eps^{-3/2})$ oracle complexity \citep{luo2022}.
More recently, \citet{ma2026} propose a single-loop adaptive cubic-quadratic regularized Newton method (ACQRN) for the penalty $h(\bx,\by)=f(\bx,\by)+(1/\mu)\norm{\nabla_yf(\bx,\by)}^2$. 
Their method takes $\cO(\sqrt{\kappa\rho}\,\epsilon^{-3/2})$ complexity to find approximate SOSP of $h(\cdot,\cdot)$, 
which leads to the $\cO(\sqrt{\rho}\kappa^{2}\,\epsilon^{-3/2})$ complexity  for finding an approximate SOSP of $\Phi(\cdot)$.
Although the $\eps^{-3/2}$ dependence is optimal by the lower bound for smooth nonconvex minimization \citep{carmon2020lower}, 
it still remains open whether the condition number dependence is optimal.

In this paper, we propose Envelope Cubic Regularized Newton (ECRN), 
which finds approximate SOSP of $\Phi(\cdot)$ with $\tO(\sqrt\rho\,\kappa^{3/4}\eps^{-3/2})$ second-order oracle complexity. 
The method exploits a sharper Hessian-Lipschitz bound for the primal envelope in~\eqref{eq:envelope}. 
In the prior literature, the Hessian-Lipschitz parameter of the envelope function is known to be the same order as that of the original function.
By taking the advantage of the minimax structure, we show that the Hessian of $\Phi_{2L}(\cdot)$ is $\mathcal{O}(\rho \kappa^{3/2})$-Lipschitz continuous,
which is better than the $\mathcal{O}(\rho\kappa^3)$-Lipschitz Hessian of $\Phi(\cdot)$.
Applying inexact cubic regularization to $\Phi_{2L}(\cdot)$ therefore requires $\cO(\sqrt\rho\,\kappa^{3/4}\eps^{-3/2})$ outer iterations.
Approximating $\nabla\Phi_{2L}(\cdot)$ and $\nabla^2\Phi_{2L}(\cdot)$ requires solving an unbalanced strongly-convex-strongly-concave (SC-SC) problem at each outer iteration. 
To solve these subproblems efficiently, we propose NewtonTrack, which constructs a good initialization from previously computed approximate solutions 
and applies Newton's method sequencially.
We also provide lower bound for second-order methods to solve nonconvex strongly-concave problem, which shows that the proposed ECRN method is near-optimal.

\section{Preliminaries and Related Work}\label{sec:prelim}
We use $\norm{\cdot}$ for the Euclidean norm of vectors and the induced operator norm of matrices. The matrix $\bI_d$ denotes the $d$-dimensional identity.
We impose the following assumptions throughout the paper.

\begin{assumption}\label{ass:smooth}
The function $f:\R^{d_x}\times\R^{d_y}\to\R$ is twice continuously differentiable. For all $\bw,\bv\in\R^{d_x+d_y}$, we assume that
\begin{equation}\label{eq:smooth}
\norm{\nabla^2f(\bw)}\le L~~~\text{and}~~~
\norm{\nabla^2f(\bw)-\nabla^2f(\bv)}\le\rho\norm{\bw-\bv},
\end{equation}
where $L,\rho>0$.
\end{assumption}
\begin{assumption}\label{ass:concave}
For every $\bx$, the function $f(\bx,\cdot)$ is $\mu$-strongly concave, where $0<\mu\le L$. We write $\kappa=L/\mu$.
\end{assumption}
\begin{assumption}\label{ass:gap}
The primal function $\Phi(\cdot)$ is bounded below, and the initial point $\bz_0$ satisfies $\Phi(\bz_0)-\inf_{\bz}\Phi(\bz)\le\Delta$ for some $\Delta>0$.
\end{assumption}

We also introduce the following definition of approximate first-order and second-order stationary points of $\Phi(\cdot)$.
\begin{definition}\label{def:sosp}
For $\eps,\eta>0$, a point $\hat\bx$ is an $\eps$-first-order stationary point of $\Phi(\cdot)$ if $\norm{\nabla\Phi(\hat\bx)}\le\eps$. 
It is an $(\eps,\eta)$-second-order stationary point if, in addition,
$
\nabla^2\Phi(\hat\bx)\succeq-\eta\bI_{d_x}.
$
\end{definition}

\subsection{Inexact Cubic Regularization Method}\label{sec:inexact-crn}
We recall the inexact variant of cubic reguarlized Newton (CRN) method~\citep{nesterov2006} to solve the minimization problem
\[
\min_{\bz\in\R^d}g(\bz),
\]
where $g:\R^d\to\R$ is twice continuously differentiable and its Hessian is $\rho_g$-Lipschitz continuous for some $\rho_g>0$. 
At each iterate $\bz_k$, we can obtain an approximate gradient $\widehat{\bg}_k$ and a symmetric approximate Hessian $\widehat{\bH}_k$ satisfying
\begin{equation}\label{eq:model-contract}
\norm{\widehat{\bg}_k-\nabla g(\bz_k)}\le\frac{\eps}{128}
\quad\text{and}\quad
\norm{\widehat{\bH}_k-\nabla^2g(\bz_k)}\le\frac{\sqrt{\rho_g\eps}}{128}.
\end{equation}
The following lemma shows the inexact CRN method presented in Algorithm~\ref{alg:inexact-crn} maintain the $\cO(\epsilon^{-3/2})$ iteration complexity of CRN method.

\begin{algorithm}[htbp]
\caption{Inexact Cubic Regularized Newton (Inexact CRN)}\label{alg:inexact-crn}
\begin{algorithmic}[1]
\State \textbf{Input:} initial point $\bz_0$, accuracy $\eps>0$, constant $\rho_g>0$, and a derivative routine satisfying~\eqref{eq:model-contract}.
\For{$k=0,1,2,\ldots$}
\State Obtain $\widehat{\bg}_k$ and $\widehat{\bH}_k$ satisfying~\eqref{eq:model-contract}.
\State $\bs_k\in\argmin_{\bs}\{\ip{\widehat{\bg}_k}{\bs}+\frac{1}{2}\bs^\top\widehat{\bH}_k\bs+\frac{\rho_g}{3}\norm{\bs}^3\}$.
\State $\bz_{k+1}\gets\bz_k+\bs_k$.
\If{$\norm{\bs_k}\le\sqrt{\frac{\eps}{8\rho_g}}$}
\State \textbf{return} $\bz_{k+1}$.
\EndIf
\EndFor
\end{algorithmic}
\end{algorithm}

\begin{lemma}[{cf.\ \citet[Lemmas~2--3]{wang2019inexact}}]\label{lem:cubic}
Let $g:\R^d\to\R$ be twice continuously differentiable with a $\rho_g$-Lipschitz Hessian, where $\rho_g>0$, and suppose $g(\bz_0)-\inf_{\bz}g(\bz)\le\Delta$. If the derivative routine satisfies~\eqref{eq:model-contract} at every iterate, then Algorithm~\ref{alg:inexact-crn} returns $\bz_{K+1}$ after $K+1$ iterations, where
\begin{equation}\label{eq:cubic-sum}
K\le4\cdot8^{\frac{3}{2}}\Delta\sqrt{\rho_g}\,\eps^{-\frac{3}{2}}
\quad\text{and}\quad
\sum_{k=0}^{K-1}\norm{\bs_k}^3\le\frac{4\Delta}{\rho_g}.
\end{equation}
Its output satisfies
\begin{equation}\label{eq:cubic-stop}
\norm{\nabla g(\bz_{K+1})}<\frac{\eps}{2}
\quad\text{and}\quad
\nabla^2g(\bz_{K+1})\succeq-\frac{3}{4}\sqrt{\rho_g\eps}\,\bI_d.
\end{equation}
\end{lemma}

\section{The Primal Envelope and the Regularized Saddle Problem}\label{sec:envelope-saddle}
We first establish the properties of the primal envelope and its related saddle point problem in Section~\ref{sec:regularized-envelope}. 
We then introduce efficient methods to solve the regularized minimax problem in Section~\ref{sec:saddle-solvers}.

\subsection{Properties of the Regularized Saddle Problem and the Primal Envelope}\label{sec:regularized-envelope}
For $\bz\in\R^{d_x}$, consider the primal envelope $\Phi_{2L}(\cdot)$ defined in~\eqref{eq:envelope}. Write $\bw=(\bu,\by)$ and define the regularized objective by
\begin{equation}\label{eq:regularized-objective}
f_{\bz}(\bw):=f(\bu,\by)+L\norm{\bu-\bz}^2.
\end{equation}
For each fixed $\bz$, we denote the unique saddle point of \eqref{eq:regularized-objective} by
\begin{equation}\label{eq:saddle-map}
\bw^*(\bz):=\bigl(\bu^*(\bz),\by^*(\bz)\bigr),
\end{equation}
where $\bu^*(\bz)$ is the minimizer in~\eqref{eq:envelope} and $\by^*(\bz):=\argmax_{\by}f(\bu^*(\bz),\by)$. In particular, we have
\[
\Phi_{2L}(\bz)=f_{\bz}(\bw^*(\bz)).
\]
The gradient and Hessian of $f_{\bz}(\bw)$ are
\begin{equation}\label{eq:regularized-derivatives}
\begin{aligned}
&\nabla f_{\bz}(\bw)=
\begin{pmatrix}\nabla_uf(\bw)+2L(\bu-\bz)\\\nabla_yf(\bw)\end{pmatrix}~~~\text{and}~~~\nabla^2f_{\bz}(\bw)=
\begin{pmatrix}
\nabla^2_{uu}f(\bw)+2L\bI_{d_x}&\nabla^2_{uy}f(\bw)\\
\nabla^2_{yu}f(\bw)&\nabla^2_{yy}f(\bw)
\end{pmatrix}.
\end{aligned}
\end{equation}
The following proposition gives the basic properties of $f_{\bz}(\cdot)$.

\begin{proposition}[{cf.\ \citet[Lemma~1.1 and Theorem~1.2]{grimmer2023landscape}}]\label{prop:primal}
Under Assumptions~\ref{ass:smooth}--\ref{ass:concave}, the regularized objective in~\eqref{eq:regularized-objective} has the following properties.
\begin{enumerate}[label=(\roman*),leftmargin=2em,itemsep=3pt]
\item For every $\bz\in\R^{d_x}$, $f_{\bz}(\cdot)$ is $L$-strongly convex in $\bu$ and $\mu$-strongly concave in $\by$, and has the unique saddle point~\eqref{eq:saddle-map}. Its Hessian is nonsingular, with
\begin{equation}\label{eq:regularized-inverse-bound}
\norm{\bigl(\nabla^2f_{\bz}(\bw)\bigr)^{-1}}\le \mu^{-1}.
\end{equation}
\item For every $\bz\in\R^{d_x}$ and every $\bw,\bv\in\R^{d_x+d_y}$, we have
\begin{equation}\label{eq:regularized-hessian-bounds}
\begin{aligned}
\norm{\nabla^2f_{\bz}(\bw)}\le 3L~~~\text{and}~~~\norm{\nabla^2f_{\bz}(\bw)-\nabla^2f_{\bz}(\bv)}\le\rho\norm{\bw-\bv}.
\end{aligned}
\end{equation}
\end{enumerate}
\end{proposition}

The next proposition gives the properties of the primal envelope.

\begin{proposition}[{cf.\ \citet[Lemma~17]{lin2020near} and \citet[Section~1.3]{grimmer2023landscape}}]\label{prop:envelope-basic}
Under Assumptions~\ref{ass:smooth}--\ref{ass:concave}, the envelope $\Phi_{2L}(\cdot)$ is twice continuously differentiable, with
\begin{equation}\label{eq:envelope-basic}
\nabla\Phi_{2L}(\bz)=2L\bigl(\bz-\bu^*(\bz)\bigr)=\nabla\Phi(\bu^*(\bz))
\end{equation}
and
\begin{equation}\label{eq:envelope-hessian}
\nabla^2\Phi_{2L}(\bz)=2L\bI_{d_x}-4L^2\bM(\bw^*(\bz))^{-1},
\end{equation}
where $\bM(\bw)$ is defined by
\begin{equation}\label{eq:schur-complement}
\bM(\bw):=\nabla^2_{uu}f(\bw)+2L\bI_{d_x}
-\nabla^2_{uy}f(\bw)\bigl(\nabla^2_{yy}f(\bw)\bigr)^{-1}\nabla^2_{yu}f(\bw).
\end{equation}
For every $\bw\in\R^{d_x+d_y}$, we have $\bM(\bw)\succeq L\bI_{d_x}$. The envelope has the following properties.
\begin{enumerate}[label=(\roman*),leftmargin=2em,itemsep=3pt]
\item Its infimum is $\inf_{\bz}\Phi_{2L}(\bz)=\inf_{\bz}\Phi(\bz)$. 
Under Assumption~\ref{ass:gap}, we have $\Phi_{2L}(\bz_0)-\inf_{\bz}\Phi_{2L}(\bz)\le\Delta$.
\item If $\bz$ is an $(\eps,\eta)$-second-order stationary point of $\Phi_{2L}(\cdot)$, then $\bu^*(\bz)$ is an $(\eps,2L\eta/(2L+\eta))$-second-order stationary point of $\Phi(\cdot)$, and hence also an $(\eps,\eta)$-second-order stationary point.
\end{enumerate}
\end{proposition}
Proposition~\ref{prop:envelope-basic} shows that it suffices to find a second-order stationary point of $\Phi_{2L}(\cdot)$.
The following lemma gives the regularity of the saddle-point map.
\begin{lemma}\label{lem:saddle-map}
Under Assumptions~\ref{ass:smooth}--\ref{ass:concave}, the saddle-point map $\bw^*(\cdot)$ is continuously differentiable and for every $\bz,\bz'\in\R^{d_x}$, it holds that
\begin{equation}\label{eq:global-displacement}
\norm{\bw^*(\bz)-\bw^*(\bz')}
\le \sqrt{4+\kappa}\norm{\bz-\bz'}.
\end{equation}
Moreover, its distance from any point $\bw\in\R^{d_x+d_y}$ is bounded by the gradient norm:
\begin{equation}\label{eq:residual-to-distance}
\norm{\bw-\bw^*(\bz)}\le \mu^{-1}\norm{\nabla f_{\bz}(\bw)}.
\end{equation}
\end{lemma}
The following lemma bounds the Hessian Lipschitz constant of $\Phi_{2L}(\cdot)$.
\begin{lemma}\label{lem:envelope}
Under Assumptions~\ref{ass:smooth}--\ref{ass:concave}, the Hessian of $\Phi_{2L}(\cdot)$ is $\widetilde{\rho}$-Lipschitz continuous, where $\widetilde{\rho}:=\rho(4+\kappa)^{3/2}$. Specifically, for every $\bz,\bz'\in\R^{d_x}$,
\begin{equation}\label{eq:H}
\norm{\nabla^2\Phi_{2L}(\bz)-\nabla^2\Phi_{2L}(\bz')}
\le\widetilde{\rho}\norm{\bz-\bz'}.
\end{equation}
\end{lemma}
\begin{proof}
We write $\bP:=[2L\bI_{d_x},\bzero_{d_x\times d_y}]^\top$. 
Block elimination gives
\[
4L^2\bM(\bw)^{-1}=\bP^\top\bigl(\nabla^2f_{\bz}(\bw)\bigr)^{-1}\bP.
\]
Thus, for any points $\bw,\bv$, the inverse-difference identity gives
\begin{equation}\label{eq:envelope-hessian-difference}
\begin{aligned}
&4L^2\bigl(\bM(\bv)^{-1}-\bM(\bw)^{-1}\bigr)\\
&\quad=\bP^\top\bigl(\nabla^2f_{\bz}(\bw)\bigr)^{-1}
\bigl(\nabla^2f_{\bz}(\bw)-\nabla^2f_{\bz}(\bv)\bigr)
\bigl(\nabla^2f_{\bz}(\bv)\bigr)^{-1}\bP.
\end{aligned}
\end{equation}
Using~\eqref{eq:envelope-hessian-difference} and symmetry of the Hessians, we have
\begin{equation}\label{eq:schur-inverse-lipschitz}
\begin{aligned}
&4L^2\norm{\bM(\bv)^{-1}-\bM(\bw)^{-1}}\\
&\quad\le
\norm{\bigl(\nabla^2f_{\bz}(\bw)\bigr)^{-1}\bP}
\norm{\nabla^2f_{\bz}(\bw)-\nabla^2f_{\bz}(\bv)}
\norm{\bigl(\nabla^2f_{\bz}(\bv)\bigr)^{-1}\bP}\\
&\quad\stackrel{\eqref{eq:T-bound}}{\le}
(4+\kappa)\norm{\nabla^2f_{\bz}(\bw)-\nabla^2f_{\bz}(\bv)}\\
&\quad\stackrel{\eqref{eq:regularized-hessian-bounds}}{\le}
\rho (4+\kappa)\norm{\bw-\bv}.
\end{aligned}
\end{equation}
Apply \eqref{eq:schur-inverse-lipschitz} with $\bw=\bw^*(\bz)$ and $\bv=\bw^*(\bz')$ yields
\[
\begin{aligned}
&\norm{\nabla^2\Phi_{2L}(\bz)-\nabla^2\Phi_{2L}(\bz')}\overset{\eqref{eq:envelope-hessian}}{\leq} 4L^2\norm{\bM(\bv)^{-1}-\bM(\bw)^{-1}}
\\
&~~~\overset{\eqref{eq:schur-inverse-lipschitz}}{\le}\rho (4+\kappa)\norm{\bw^*(\bz)-\bw^*(\bz')}
\overset{\eqref{eq:global-displacement}}{\le}\rho (4+\kappa)^{\frac{3}{2}}\norm{\bz-\bz'}.
\end{aligned}
\]
This proves~\eqref{eq:H}.
\end{proof}

\subsection{Solving the Regularized Saddle Problem}\label{sec:saddle-solvers}
We consider the regularized minimax problem
\begin{equation}\label{obj:minimax}
\min_{\bu\in\R^{d_x}}\max_{\by\in\R^{d_y}} f_{\bz}(\bw)
=\min_{\bu\in\R^{d_x}}\max_{\by\in\R^{d_y}}\bigl\{f(\bu,\by)+L\norm{\bu-\bz}^2\bigr\}.
\end{equation}
By Proposition~\ref{prop:primal}, its objective is $L$-strongly convex in $\bu$, $\mu$-strongly concave in $\by$, and has a $3L$-Lipschitz gradient. We seek an approximate solution $\hat{\bw} = (\hat{\bu},\hat{\by})$ such that
\begin{equation}\label{eq:inner-target}
\norm{\nabla f_{\bz}(\hat\bw)}\le\bar\eps.
\end{equation}
\subsubsection{Existing First- and Second-Order Methods}\label{sec:existing-solvers}
We recall the existing methods to solve SC-SC minimax problem~\eqref{obj:minimax}.
For the first-order method, we use the single-loop predictor--corrector extragradient method with damped momentum (PCE-DM) in Algorithm~\ref{alg:pce-dm} \cite{zhang2026near}, which can well handle the unbalanced condition-number between strongly-convexity and strongly-concavity of \eqref{obj:minimax} with a simple single-loop update.
We introduce its oracle complexity in the following lemma.

\begin{algorithm}[htbp]
\caption{$\textnormal{\textsc{PCE-DM}}(\bw_{\rm init};\bz,\bar\eps)$}\label{alg:pce-dm}
\begin{algorithmic}[1]
\State \textbf{Input:} initial point $\bw_{\rm init}=(\bu_{\rm init},\by_{\rm init})$, point $\bz$, and tolerance $\bar\eps>0$.
\State $h\gets\frac{2}{45L}$, $\alpha\gets\frac{8}{8+h\sqrt{L\mu}}$, and $\beta\gets \frac{hL}{8}$.
\State $\bw=(\bu,\by)\gets\bw_{\rm init}$ and $\bv\gets\bzero_{d_y}$.
\While{$\norm{\nabla f_{\bz}(\bw)}>\bar\eps$}
\State $\widetilde{\bu}\gets\bu-h\nabla_u f_{\bz}(\bw)$.
\State $\widetilde{\by}\gets\by+h\bigl(\nabla_y f_{\bz}(\bw)+\bv\bigr)$.
\State $\widetilde{\bw}\gets(\widetilde{\bu},\widetilde{\by})$.
\State $\widetilde{\bv}\gets\alpha\bv+\beta\nabla_y f_{\bz}(\widetilde{\bw})$.
\State $\bu\gets\bu-h\nabla_u f_{\bz}(\widetilde{\bw})$.
\State $\by\gets\by+h\bigl(\nabla_y f_{\bz}(\widetilde{\bw})+\widetilde{\bv}\bigr)$.
\State $\bw\gets(\bu,\by)$.
\State $\bv\gets\alpha\bv+\beta\nabla_y f_{\bz}(\bw)$.
\EndWhile
\State \textbf{return} $\bw$.
\end{algorithmic}
\end{algorithm}
\begin{lemma}[{cf.\ \citet[Algorithm~1 and Theorem~3.7]{zhang2026near}}]\label{lem:pce-dm}
Under Assumptions~\ref{ass:smooth}--\ref{ass:concave}, let $D_{\rm init}:=\norm{\bw_{\rm init}-\bw^*(\bz)}$. When the initial gradient norm exceeds $\bar\eps$, Algorithm~\ref{alg:pce-dm} reaches~\eqref{eq:inner-target} within $T_{\rm PCE}(\bar\eps)$ iterations and at most $2T_{\rm PCE}(\bar\eps)+1$ gradient evaluations, where
\begin{equation}\label{eq:pce-horizon}
T_{\rm PCE}(\bar\eps):=\Bigl\lceil720\sqrt\kappa
\log\frac{2592L^2\kappa^{\frac{3}{2}}D_{\rm init}^2}{\bar\eps^2}\Bigr\rceil.
\end{equation}
\end{lemma}
Next, we recall Newton's method in Algorithm~\ref{alg:newton} and its complexity.

\begin{algorithm}[htbp]
\caption{$\textnormal{\textsc{Newton}}(\bw_{\rm init};\bz,\bar\eps)$}\label{alg:newton}
\begin{algorithmic}[1]
\State \textbf{Input:} initial point $\bw_{\rm init}$, point $\bz$, and tolerance $\bar\eps>0$.
\State $\bw\gets\bw_{\rm init}$.
\While{$\norm{\nabla f_{\bz}(\bw)}>\bar\eps$}
\State $\bw\gets\bw-\bigl(\nabla^2 f_{\bz}(\bw)\bigr)^{-1}\nabla f_{\bz}(\bw)$.
\EndWhile
\State \textbf{return} $\bw$.
\end{algorithmic}
\end{algorithm}

\begin{lemma}[{cf.\ the local Newton analysis of \citet{huang2022cubic}}]\label{lem:local-newton}
Under Assumptions~\ref{ass:smooth}--\ref{ass:concave}, suppose
\begin{equation}\label{eq:local-newton-condition}
\norm{\bw_{\rm init}-\bw^*(\bz)}\le\frac{\mu}{12\rho}.
\end{equation}
When the initial gradient norm exceeds $\bar\eps$, Algorithm~\ref{alg:newton} reaches~\eqref{eq:inner-target} using at most $T_{\rm N}(\bar\eps)$ Hessian evaluations and $T_{\rm N}(\bar\eps)+1$ gradient evaluations, where
\begin{equation}\label{eq:local-newton-horizon}
T_{\rm N}(\bar\eps):=
\left\lceil\log_2\log_2\frac{6L\mu}{\rho\bar\eps}\right\rceil.
\end{equation}
\end{lemma}

\subsubsection{The NewtonTrack Method}\label{sec:newton-track}
\label{sec:newton-track-method}
The first-order methods converge from an arbitrary starting point, but their costs depend polynomially on the condition number. 
Newton's method requires only a doubly logarithmic number of local iterations, but can only converge locally.

In this section, we solve the regularized minimax problem~\eqref{obj:minimax}
using an approximate saddle point $\bar{\bw}$ of $f_{\bar{\bz}}(\cdot)$ from  a reference point $\bar{\bz}$, with $\norm{\nabla f_{\bar{\bz}}(\bar{\bw})}\le\bar\eps$. 
The following lemma shows that $\bar{\bw}$ satisfies the local condition~\eqref{eq:local-newton-condition} of Newton's method when $\bar\eps$ and $\norm{\bz-\bar{\bz}}$ are sufficiently small.

\begin{lemma}\label{lem:warm-start}
Under Assumptions~\ref{ass:smooth}--\ref{ass:concave}, choose $R,\bar\eps$ such that
\begin{equation}\label{eq:tracking-conditions}
0<\bar\eps\le\frac{\mu^2}{48\rho}
\quad\text{and}\quad
0<R\le\frac{\mu}{16\rho\sqrt{4+\kappa}}.
\end{equation}
If $\norm{\nabla f_{\bar{\bz}}(\bar{\bw})}\le\bar\eps$ and $\norm{\bz-\bar{\bz}}\le R$, then
\begin{equation}\label{eq:warm-radius}
\norm{\bar{\bw}-\bw^*(\bz)}\le\frac{\mu}{12\rho}.
\end{equation}
\end{lemma}
\begin{proof}
By~\eqref{eq:residual-to-distance} and~\eqref{eq:global-displacement}, we have
\[
\begin{aligned}
\norm{\bar{\bw}-\bw^*(\bz)}
&\le\norm{\bar{\bw}-\bw^*(\bar{\bz})}+\norm{\bw^*(\bar{\bz})-\bw^*(\bz)}\\
&\le \frac{\bar\eps}{\mu}+\sqrt{4+\kappa}R
\le\frac{\mu}{48\rho}+\frac{\mu}{16\rho}
=\frac{\mu}{12\rho}.
\end{aligned}
\]
\end{proof}
When $\norm{\bz-\bar{\bz}}>R$, we divide the segment from $\bar{\bz}$ to $\bz$ into $m=\lceil\norm{\bz-\bar{\bz}}/R\rceil$ pieces. We set $\bxi_0=\bar{\bz}$ and $\bxi_j=\bar{\bz}+(j/m)(\bz-\bar{\bz})$ for $j=1,\ldots,m$, so that $\bxi_m=\bz$ and $\norm{\bxi_j-\bxi_{j-1}}\le R$. We approximately solve the regularized saddle problems in sequence, using the approximate solution of $f_{\bxi_{j-1}}(\cdot)$ to initialize Newton's method for the problem with objective $f_{\bxi_j}(\cdot)$. We call this method NewtonTrack and formally present it in Algorithm~\ref{alg:newton-track}.
The following lemma provides the oracle complexity of NewtonTrack.
\begin{algorithm}[htbp]
\caption{$\textnormal{\textsc{NewtonTrack}}(\bar{\bw};\bar{\bz},\bz,\bar\eps,R)$}\label{alg:newton-track}
\begin{algorithmic}[1]
\State \textbf{Input:} points $\bar{\bz},\bz$, an approximate saddle point $\bar{\bw}$ of $f_{\bar{\bz}}(\cdot)$ with $\norm{\nabla f_{\bar{\bz}}(\bar{\bw})}\le\bar\eps$, and $\bar\eps,R$ satisfying~\eqref{eq:tracking-conditions}.
\State $\bw_0\gets\bar{\bw}$, $\bxi_0\gets\bar{\bz}$, and $m\gets\left\lceil\frac{\norm{\bz-\bar{\bz}}}{R}\right\rceil$.
\For{$j=1,\ldots,m$}
\State $\bxi_j\gets\bar{\bz}+\frac{j}{m}(\bz-\bar{\bz})$.
\State $\bw_j\gets\textnormal{\textsc{Newton}}(\bw_{j-1};\bxi_j,\bar\eps)$.
\EndFor
\State \textbf{return} $\bw_m$.
\end{algorithmic}
\end{algorithm}

\begin{lemma}\label{lem:tracking}
Under Assumptions~\ref{ass:smooth}--\ref{ass:concave}, let $R,\bar\eps$ satisfy~\eqref{eq:tracking-conditions}. If $\norm{\nabla f_{\bar{\bz}}(\bar{\bw})}\le\bar\eps$, Algorithm~\ref{alg:newton-track} returns an approximate solution $\hat{\bw}=\bw_m$ satisfying $\norm{\nabla f_{\bz}(\hat{\bw})}\le\bar\eps$. With $m=\lceil\norm{\bz-\bar{\bz}}/R\rceil$, it uses at most $m T_{\rm N}(\bar\eps)$ Hessian evaluations and $m(T_{\rm N}(\bar\eps)+1)$ gradient evaluations, where $T_{\rm N}(\bar\eps)$ is given by~\eqref{eq:local-newton-horizon}.
\end{lemma}
\begin{proof}
We prove by induction that the iterates in Algorithm~\ref{alg:newton-track} satisfy
\[
\norm{\nabla f_{\bxi_j}(\bw_j)}\le\bar\eps~~~\text{for}~~~ j=0,\ldots,m.
\]
The input condition gives the claim at $j=0$, since $\bw_0=\bar{\bw}$ and $\bxi_0=\bar{\bz}$. Suppose the claim holds at $j-1$, where $1\le j\le m$. Since $\norm{\bxi_j-\bxi_{j-1}}\le R$, Lemma~\ref{lem:warm-start} gives
\[
\norm{\bw_{j-1}-\bw^*(\bxi_j)}\le\frac{\mu}{12\rho}.
\]
Lemma~\ref{lem:local-newton} and the stopping test in Algorithm~\ref{alg:newton} therefore ensure that the returned iterate $\bw_j$ satisfies
\[
\norm{\nabla f_{\bxi_j}(\bw_j)}\le\bar\eps.
\]
This completes the induction. Since $\hat{\bw}=\bw_m$ and $\bxi_m=\bz$, we obtain
\[
\norm{\nabla f_{\bz}(\hat{\bw})}
=\norm{\nabla f_{\bxi_m}(\bw_m)}\le\bar\eps.
\]
Each Newton call uses at most $T_{\rm N}(\bar\eps)$ Hessian evaluations and $T_{\rm N}(\bar\eps)+1$ gradient evaluations. Summing over the $m$ calls proves the stated oracle bounds.
\end{proof}

\section{Envelope Cubic Regularized Newton and Its Oracle Complexity}\label{sec:algorithm}

Prior second-order methods apply inexact CRN directly to $\Phi(\cdot)$, whose Hessian Lipschitz constant is $\cO(\rho\kappa^3)$ \citep{luo2022,chen2023cubic}.
Lemma~\ref{lem:envelope} instead gives a smaller constant $\widetilde{\rho}=\rho(4+\kappa)^{3/2}$ for the primal envelope $\Phi_{2L}(\cdot)$. 
Thus, it is natural to conduct inexact CRN on the primal envelop $\Phi_{2L}(\cdot)$ instead of $\Phi(\cdot)$.

The following lemma bounds the errors in approximating the envelope derivatives and transfers second-order stationarity to the computed primal point.
\begin{lemma}\label{lem:oracle}
Under Assumptions~\ref{ass:smooth}--\ref{ass:concave}, let $\eps>0$ and set
\begin{equation}\label{eq:parameters}
\widetilde\eps:=\min\Bigl\{\frac{\mu\eps}{256L},\
\frac{\mu\sqrt{\widetilde{\rho}\eps}}{128\rho(4+\kappa)},\
\frac{\mu^2}{48\rho},\
\frac{\mu\sqrt{\widetilde{\rho}\eps}}{4\rho(1+\kappa)^3}\Bigr\}.
\end{equation}
For any point $\bw=(\bu,\by)$ satisfying $\norm{\nabla f_{\bz}(\bw)}\le\widetilde\eps$, we define
\begin{equation}\label{eq:approx-derivatives}
\widehat{\bg}:=2L(\bz-\bu)
\quad\text{and}\quad
\widehat{\bH}:=2L\bI_{d_x}-4L^2\bM(\bw)^{-1}.
\end{equation}
Then, it holds that the matrix $\widehat{\bH}$ is symmetric and the derivative errors satisfy
\begin{equation}\label{eq:model-errors}
\norm{\widehat{\bg}-\nabla\Phi_{2L}(\bz)}\le\frac{\eps}{128}
\quad\text{and}\quad
\norm{\widehat{\bH}-\nabla^2\Phi_{2L}(\bz)}\le\frac{\sqrt{\widetilde{\rho}\eps}}{128}.
\end{equation}
% Thus they satisfy~\eqref{eq:model-contract} for $g(\cdot)=\Phi_{2L}(\cdot)$ and $\rho_g=\widetilde{\rho}$.
Moreover, if $\bz$ is an $(\eps/2,3\sqrt{\widetilde{\rho}\eps}/4)$-second-order stationary point of $\Phi_{2L}(\cdot)$, then $\bu$ is an $(\eps,\sqrt{\widetilde{\rho}\eps})$-second-order stationary point of $\Phi(\cdot)$.
\end{lemma}

We obtain the required approximate saddle points using the methods in Section~\ref{sec:saddle-solvers}. 
PCE-DM (Algorithm~\ref{alg:pce-dm}) computes $\bw_0$ from an arbitrary initial point $\bw_{\rm init}$, with $\norm{\nabla f_{\bz_0}(\bw_0)}\le\widetilde\eps$. 
At each subsequent outer iteration, NewtonTrack (Algorithm~\ref{alg:newton-track}) uses the approximate saddle point $\bw_{k-1}$ of $f_{\bz_{k-1}}(\cdot)$ to compute $\bw_k$ satisfying $\norm{\nabla f_{\bz_k}(\bw_k)}\le\widetilde\eps$. 
We present the resulting Envelope Cubic Regularized Newton (ECRN) method in Algorithm~\ref{alg:envelope-newton}.

\begin{algorithm}[htbp]
\caption{Envelope Cubic Regularized Newton (ECRN)}\label{alg:envelope-newton}
\begin{algorithmic}[1]
\State \textbf{Input:} initial points $\bz_0$ and $\bw_{\rm init}$, accuracy $\eps>0$.
\State \textbf{Parameters:} $\widetilde\eps$ from~\eqref{eq:parameters} and $R=\frac{\mu}{48\rho\sqrt\kappa}$.
\State $\bw_0=(\bu_0,\by_0)\gets\textnormal{\textsc{PCE-DM}}(\bw_{\rm init};\bz_0,\widetilde\eps)$.
\For{$k=0,1,2,\ldots$}
\State $\widehat{\bg}_k\gets2L(\bz_k-\bu_k)$.
\State $\widehat{\bH}_k\gets2L\bI_{d_x}-4L^2\bM(\bw_k)^{-1}$.
\State $\bs_k\in\argmin_{\bs}\{\ip{\widehat{\bg}_k}{\bs}+\frac{1}{2}\bs^\top\widehat{\bH}_k\bs+\frac{\widetilde{\rho}}{3}\norm{\bs}^3\}$.
\State $\bz_{k+1}\gets\bz_k+\bs_k$.
\State $\bw_{k+1}=(\bu_{k+1},\by_{k+1})\gets\textnormal{\textsc{NewtonTrack}}(\bw_k;\bz_k,\bz_{k+1},\widetilde\eps,R)$.
\If{$\norm{\bs_k}\le\sqrt{\frac{\eps}{8\widetilde{\rho}}}$} \State \textbf{return} $\bu_{k+1}$. \EndIf
\EndFor
\end{algorithmic}
\end{algorithm}

\begin{theorem}\label{thm:upper}
Under Assumptions~\ref{ass:smooth}--\ref{ass:gap}, for every $0<\eps\le L^2/\widetilde{\rho}$, {\rm ECRN} returns an $(\eps,\sqrt{\widetilde{\rho}\eps})$-second-order stationary point of $\Phi(\cdot)$. 
Its total number of gradient and Hessian evaluations satisfies
\begin{equation}\label{eq:main-complexity}
N_{\rm grad}+N_{\rm Hess}
=\tO\!\left(
\Delta\sqrt\rho\,\kappa^{\frac{3}{4}}\eps^{-\frac{3}{2}}\right).
\end{equation}
\end{theorem}

\begin{proof}
Let $K$ be the index at which ECRN satisfies its stopping condition, $r_k=\norm{\bs_k}$, and $m_k=\lceil r_k/R\rceil$. Lemma~\ref{lem:pce-dm}, applied to the initialization with $\bar\eps=\widetilde\eps$, and the stopping test in Algorithm~\ref{alg:pce-dm} give
\[
\norm{\nabla f_{\bz_0}(\bw_0)}\le\widetilde\eps.
\]
The parameter choices satisfy
\[
\widetilde\eps\stackrel{\eqref{eq:parameters}}{\le}\frac{\mu^2}{48\rho}
\quad\text{and}\quad
R=\frac{\mu}{48\rho\sqrt\kappa}
\le\frac{\mu}{16\rho\sqrt{4+\kappa}}.
\]
Applying Lemma~\ref{lem:tracking} with $\bar\eps=\widetilde\eps$ and $(\bar{\bw},\bar{\bz},\bz)=(\bw_k,\bz_k,\bz_{k+1})$ inductively yields
\[
\norm{\nabla f_{\bz_k}(\bw_k)}\le\widetilde\eps~~~\text{for}~~~ k=0,\ldots,K+1.
\]
For each $k=0,\ldots,K$, Lemma~\ref{lem:oracle} gives
\[
\norm{\widehat{\bg}_k-\nabla\Phi_{2L}(\bz_k)}\stackrel{\eqref{eq:model-errors}}{\le}\frac{\eps}{128}
\quad\text{and}\quad
\norm{\widehat{\bH}_k-\nabla^2\Phi_{2L}(\bz_k)}\stackrel{\eqref{eq:model-errors}}{\le}\frac{\sqrt{\widetilde{\rho}\eps}}{128}.
\]
By Proposition~\ref{prop:envelope-basic}, we also have
\[
\Phi_{2L}(\bz_0)-\inf_{\bz}\Phi_{2L}(\bz)
=\Phi_{2L}(\bz_0)-\inf_{\bz}\Phi(\bz)
\stackrel{\eqref{eq:envelope}}{\le}\Phi(\bz_0)-\inf_{\bz}\Phi(\bz)\le\Delta.
\]
Applying Lemma~\ref{lem:cubic} with $g(\cdot)=\Phi_{2L}(\cdot)$ and $\rho_g=\widetilde{\rho}$ yields
\begin{equation}\label{eq:nonterminal-step-bounds}
K\stackrel{\eqref{eq:cubic-sum}}{\le}4\cdot8^{\frac{3}{2}}\Delta\sqrt{\widetilde{\rho}}\,\eps^{-\frac{3}{2}}
\quad\text{and}\quad
\sum_{k=0}^{K-1}r_k^3\stackrel{\eqref{eq:cubic-sum}}{\le}\frac{4\Delta}{\widetilde{\rho}}.
\end{equation}
At iteration $K$, Lemma~\ref{lem:cubic} gives
\[
\norm{\nabla\Phi_{2L}(\bz_{K+1})}\stackrel{\eqref{eq:cubic-stop}}{<}\frac{\eps}{2}
\quad\text{and}\quad
\nabla^2\Phi_{2L}(\bz_{K+1})
\stackrel{\eqref{eq:cubic-stop}}{\succeq}-\frac{3}{4}\sqrt{\widetilde{\rho}\eps}\,\bI_{d_x}.
\]
The stationarity conclusion of Lemma~\ref{lem:oracle}, applied with $\bz=\bz_{K+1}$ and $\bw=\bw_{K+1}$, yields
\[
\norm{\nabla\Phi(\bu_{K+1})}\le\eps
\quad\text{and}\quad
\nabla^2\Phi(\bu_{K+1})\succeq-\sqrt{\widetilde{\rho}\eps}\,\bI_{d_x}.
\]
Then, we bound the total oracle complexity of ECRN.
The definitions of $R$ and $\widetilde{\rho}$, together with $\eps\le L^2/\widetilde{\rho}$, give
\begin{align}
\label{eq:delta-R-bound-1}
&\frac{\sqrt{\eps}}{R\sqrt{\widetilde{\rho}}}
=\frac{\sqrt{\widetilde{\rho}\eps}}{\widetilde{\rho}R}
\le\frac{L}{\widetilde{\rho}R}
=48\left(\frac{\kappa}{4+\kappa}\right)^{\frac{3}{2}}\le48,
\end{align}
which gives
\begin{align}
\label{eq:delta-R-bound-2}
\frac{\Delta}{R\eps}
=\frac{\sqrt{\eps}}{R\sqrt{\widetilde{\rho}}}
\Delta\sqrt{\widetilde{\rho}}\,\eps^{-\frac{3}{2}}
\stackrel{\eqref{eq:delta-R-bound-1}}{\le}48\Delta\sqrt{\widetilde{\rho}}\,\eps^{-\frac{3}{2}}.
\end{align}
For $k=0,\ldots,K-1$, the stopping condition is not satisfied, so $r_k>\sqrt{\eps/(8\widetilde{\rho})}$. Thus, it holds that
\begin{equation}\label{eq:r_k_bound}
r_k=\frac{r_k^3}{r_k^2}\le\frac{8\widetilde{\rho}}{\eps}r_k^3.
\end{equation}
The number of segments in NewtonTrack method satisfies that
\begin{equation}\label{eq:m_k_bound}
m_k=\left\lceil \frac{r_k}{R}\right\rceil\le1+\frac{r_k}{R}.
\end{equation}
Summing up \eqref{eq:m_k_bound} gives
\begin{equation}\label{eq:nonterminal-segments}
\begin{aligned}
\sum_{k=0}^{K-1}m_k
&\stackrel{\eqref{eq:m_k_bound}}{\le}K+\frac{1}{R}\sum_{k=0}^{K-1}r_k\stackrel{\eqref{eq:r_k_bound}}{\le}
K+\frac{8\widetilde{\rho}}{R\eps}\sum_{k=0}^{K-1}r_k^3\\
&\stackrel{\eqref{eq:nonterminal-step-bounds}}{\le}
K+\frac{32\Delta}{R\eps}\stackrel{\eqref{eq:nonterminal-step-bounds},\eqref{eq:delta-R-bound-2}}{=}
\cO(\Delta\sqrt{\widetilde{\rho}}\,\eps^{-\frac{3}{2}}).
\end{aligned}
\end{equation}
At iteration $K$, the stopping test gives $r_K\le\sqrt{\eps/(8\widetilde{\rho})}$. Thus, we have
\[
m_K\stackrel{\eqref{eq:m_k_bound}}{\le}1+\frac{r_K}{R}
\le1+\frac{\sqrt{\eps}}{\sqrt8R\sqrt{\widetilde{\rho}}}
\stackrel{\eqref{eq:delta-R-bound-1}}{\le}1+\frac{48}{\sqrt8}.
\]
Consequently, the total number of inner solves performed by NewtonTrack is
\begin{equation}\label{eq:total-segments}
N_{\rm seg}:=\sum_{k=0}^{K}m_k
=\cO\bigl(\Delta\sqrt{\widetilde{\rho}}\,\eps^{-\frac{3}{2}}\bigr).
\end{equation}
The initialization cost from Lemma~\ref{lem:pce-dm} and the per-call evaluation counts from Lemma~\ref{lem:tracking} give the final oracle complexity
\[
\begin{aligned}
&N_{\rm grad}+N_{\rm Hess}\\
&\quad\stackrel{\eqref{eq:local-newton-horizon}}{\le}
\underbrace{\tO(\sqrt\kappa)}_{\text{PCE-DM}}+\underbrace{K+1}_{\text{outer level}}
+\underbrace{\left(2\left\lceil\log_2\log_2\frac{6L\mu}{\rho\widetilde\eps}\right\rceil+1\right)N_{\rm seg}}_{\text{inner level}}\\
&\quad\stackrel{\eqref{eq:cubic-sum}}{\le}\tO(\sqrt\kappa)+ \mathcal{O}(\Delta \sqrt{\tilde{\rho}}\epsilon^{-\frac{3}{2}})
+\cO\!\left(N_{\rm seg}\log\log\frac{L\mu}{\rho\widetilde\eps}\right)\\
&\quad \stackrel{\eqref{eq:total-segments}}{\le} \tilde{O}(\Delta\sqrt\rho\,\kappa^{\frac{3}{4}}\eps^{-\frac{3}{2}} ).
\end{aligned}
\]

\end{proof}

\section{Lower Complexity Bound} \label{sec:lower}
We next show that the dependence in Theorem~\ref{thm:upper} cannot be improved for deterministic second-order algorithms formally defined as follows.
\begin{definition} \label{dfn:so-alg-class}
Let $\sO(\bw) = (\nabla f(\bw), \nabla^2 f(\bw))$ be the second-order oracle.
A deterministic second-order algorithm \texttt{A} consists of a sequence of measurable mappings $\{ \texttt{A}^{(t)} \}_{t \in \sN}$ such that $\texttt{A}^{(t)}$ takes in the first $t-1$ oracle responses  to produce the $t$-th query,  \textit{i.e.}, \texttt{A} generates a sequence $\{\bw^{(t)}\}_{t \in \sN}$ with
\begin{align*}
    \bw^{(t)} = \texttt{A}^{(t)} ( \sO(\bw^{(0)}), \sO(\bw^{(1)}), \cdots, \sO(\bw^{(t-1)}) ).
\end{align*}
\end{definition}
To prove lower bounds for the above algorithm class, we construct an objective $f(\bx,\by)$ by embedding a Fenchel conjugate, such that the primal objective $\Phi(\bx) = \max_{\by \in \R^{d_y}} f(\bx,\by)$ is equivalent to the smooth zero-chain of \citet{carmon2020lower} rescaled by a factor $\kappa^{3/2}$. Then we can reuse the analysis in \citet{carmon2020lower} and conclude a lower bound of $\Omega( \kappa^{3/4} \eps^{-3/2})$. The rest of this section proves this lower bound following this idea. 

\subsection{Recalling the Nonconvex Zero-Chain}
The main technique we use to prove lower bounds is to construct second-order zero-chain functions \citep{nesterov2018lectures,carmon2020lower}, which are differentiable functions satisfying: 
for every \(t\in[T]\) and every $\bx \in \R^T$, we have
\begin{align} \label{eq:dfn-zero-chain}
    \supp(\bx) \subseteq \{1,\cdots,t-1 \} ~~ \Longrightarrow ~~ 
    \begin{cases}
    \supp(\nabla f(\bx)) \subseteq \{1,\cdots,t \}; \\
    \supp(\nabla^2 f(\bx)) \subseteq \{1, \cdots, t\}^2.
    \end{cases}
\end{align}
Functions with the above properties restrict the information flow by forcing any second-order algorithm to discover and activate coordinates sequentially. Based on this idea, \citet{carmon2020lower} constructed the smooth zero-chain
\begin{equation}\label{eq:base-chain}
\bar h_T(\bx)=-\Psi(1)\Gamma(x_1)
+\sum_{i=2}^T\left[\Psi(-x_{i-1})\Gamma(-x_i)-\Psi(x_{i-1})\Gamma(x_i)\right],
\end{equation}
where
\[
\Psi(t)=\begin{cases}0&t\le\frac{1}{2}\\ \exp(1-(2t-1)^{-2})&t>\frac{1}{2}\end{cases}
~~~\text{and}~~~
\Gamma(t)=\sqrt e\int_{-\infty}^t\exp(-\frac{s^2}{2})\,ds.
\]
The main properties of this hard instance are summarized in the following lemma. Briefly speaking, this function is a second-order zero-chain, and any point whose last coordinate is zero has a large gradient norm. Consequently, to find an $\eps$-stationary point, an algorithm must make at least $T$ queries to reach the final coordinate that triggers a sufficiently small gradient norm.

\begin{samepage}
\begin{lemma}[{\citet[Section 4.1]{carmon2020lower} and \citet[Lemma 2]{arjevani2023lower}}] \label{lem:carmon}
There exists numerical constants $G,\ell_1,\ell_2>0$ such that:
\begin{enumerate}[label=(\roman*),leftmargin=2em,itemsep=5pt,parsep=5pt,topsep=6pt]
\item $\bar h_T(\bzero)-\inf_{\bx}\bar h_T(\bx)\le12T$.
\item $\norm{\nabla\bar h_T(\bx)}_\infty\le G$ and $\norm{\nabla^2\bar h_T(\bx)}\le\ell_1$.
\item $\nabla^2\bar h_T$ is $\ell_2$-Lipschitz.
\item If $|x_j|<1/2$ for all $j\ge i$, then in a neighborhood of $\bx$, the function depends only on coordinates $1,\ldots,i$.
\item If $|x_T|<1$, then $\norm{\nabla\bar h_T(\bx)}>1$.
\end{enumerate}    
\end{lemma}
\end{samepage}
For any constant $p \in \sN_+$, this hard instance can give a lower bound of $\Omega( \eps^{-p/(p+1)})$ for $p$th-order deterministic methods for finding $\eps$-stationary points of smooth nonconvex functions \citep{carmon2020lower}. In particular, for $p=2$, it leads to a tight lower bound of $\Omega(\eps^{-3/2})$ that matches the theoretical guarantee of the celebrated cubic regularized Newton (CRN) method \citep{nesterov2006}. In the following, we build upon this nonconvex function and construct a nonconvex--strongly-concave (NC--SC) function to show tight lower bounds for second-order minimax optimization.

% \begin{proof}
% The coordinatewise bound in (ii) also follows directly from \eqref{eq:base-chain}: each coordinate appears in at most two adjacent summands, and the functions and derivatives in those summands are uniformly bounded. Property (iv) follows because $\Psi$ is identically zero on a neighborhood of each tail argument. These observations make the required coordinatewise and local statements explicit.    
% \end{proof}
\subsection{Constructing a Hard Instance by Fenchel Conjugation}

To obtain lower complexity bounds for NC--SC minimax optimization from NC hard instances, a naive idea is to generalize the same technique for constructing first-order lower bounds in \citet{zhang2021complexity} to second order.
The construction of \citet{zhang2021complexity} leverages the identity 
\[
\max_{\by \in \R^{d_y}} \left\{\Upsilon(\bx,\by) \triangleq \lambda \langle \bB \bx, \by \rangle - \frac{1}{2} \| \by \|^2 \right\} = \frac{\lambda^2}{2} \| \bB \bx \|^2,
\]
and construct the coupling term in $\Upsilon(\bx,\by)$ according to the left-hand side above. 
While this simple coupling works for first-order lower bounds, it relies heavily on the fact that the first-order nonconvex hard instance can be decomposed into a quadratic component and separable nonconvex terms \citep{carmon2021lower}. When moving to second-order lower bounds, the hard instance $\bar h_T(\bx)$ in Eq. (\ref{eq:base-chain}) does not have such a simple decomposition and therefore requires new techniques. To overcome this limitation, we introduce a more principled approach based on the Fenchel conjugate. Formally, we construct a base hard instance of NC--SC minimax problems as
% \textit{conjugate embedding}. Instead of forcing an explicit bilinear coupling, we directly construct the minimax objective $f(\vx, \vy)$ by embedding the Fenchel conjugate of a composite primal objective. This guarantees that the implicit primal function $\Phi(\vx) = \max_{\vy} f(\vx, \vy)$ exactly reduces to the target non-convex smooth zero-chain $\bar{h}_T$ up to appropriate scaling, while naturally preserving the required $L$-smoothness and $\mu$-strong concavity conditions for $f$.
\begin{equation}\label{eq:hard-f}
f(\bx,\by)=-\frac{L}{4}\norm{\bx}^2+\beta\ip{\bx}{\by}-Q^*(\beta \by).
\end{equation}
Here $Q^*(\bs) \triangleq \sup_{\bx}\{\ip{\bs}{\bx}-Q(\bx)\}$ is the Fenchel conjugate of the function $Q(\cdot)$ which is defined by 
\begin{equation}\label{eq:lower-construction}
Q(\bx)= \alpha \bar h_T(\frac{\bx}{\sigma})+\frac{L}{4} \norm{\bx}^2,
\end{equation}
where $\bar h_T(\bx)$ is the second-order zero-chain defined in Eq. (\ref{eq:base-chain}).
Finally, for $c_\eps  =\alpha\ell_1/\sigma^2 \le L/8$, we choose the parameter as $\beta=\sqrt{\mu(L/2+c_\eps)}$. By the Fenchel--Moreau biconjugation theorem, we have $Q^{**}(\bx) = Q(\bx)$. Therefore, the primal function in our construction is exactly $\alpha \bar h_T(\bx/\sigma)$.

\begin{observation} \label{observation-phi}
We have $\Phi(\bx) \triangleq \max_{\by \in \R^T} f(\bx,\by) = \alpha \bar h_T(\bx/\sigma) $.
\end{observation}

Let $\rho_h \triangleq \alpha\ell_2/\sigma^3$ be the Hessian Lipschitz constant of the function $\Phi$ indicated by Lemma~\ref{lem:carmon}. It in turn provides the upper bound of Hessian Lipschitz constant of the conjugate function $\psi(\by) \triangleq Q^*(\beta \by)$ in our hard instance $f(\bx,\by)$.

\begin{lemma} \label{lem:pro-psi}
If $c_\eps\le L/8$, the function $\psi(\by) = Q^*(\beta \by)$ satisfies
\begin{equation}\label{eq:psi-bounds-main}
\mu \bI_T \preceq\nabla^2\psi(\by)\preceq\frac{5}{3}\mu \bI_T~~~\text{and}~~~
\Lip(\nabla^2\psi(\by))\le\frac{\beta^3\rho_h}{(\frac{L}{2}-c_\eps)^3},
\end{equation}
where $\Lip(\,\cdot\,)$ denotes the Lipschitz constant of a matrix.
\end{lemma}
% However, the construction of \citet{zhang2021complexity} leverages the identity 
% \[
% \max_{\vy \in \sR^{d_y}} \lambda \langle \mB \vx, \vy \rangle - \frac{1}{2} \| \vy \|^2 = \frac{\lambda^2}{2} \| \mB \vx \|^2.
% % \]
% The remaining issue is information preservation. A conjugate is defined by a global maximization, so a zero-chain property cannot simply be assumed to survive conjugation. We use a uniform coordinatewise gradient bound for $h$ to prove that sufficiently small tail coordinates of $y$ produce small tail coordinates in the conjugate maximizer. The tail then decouples into a known quadratic. This yields the second-order zero-chain property needed by a resisting oracle.

In the second useful lemma, we show that the zero-chain property is preserved for the conjugate function. The proof uses Lemma \ref{lem:carmon}(iv) and decouples the tail zero coordinates into a quadratic, which allows the exact calculation of the conjugate maximization locally.

\begin{lemma} \label{lem:zero-chain}
If $ G \alpha \le L\sigma^2/4$ and $\supp(\by)\subseteq\{1,\ldots,k\}$, then
\begin{equation}\label{eq:oracle-support}
\supp(\nabla\psi(\by))\subseteq\{1,\ldots,k+1\}~~~\text{and}~~~
\supp \left(\nabla^2\psi(\by)-\frac{2\beta^2}{L} \bI_T\right)
\subseteq \{1,\ldots,k+1\}^2.
\end{equation}
\end{lemma}

Therefore, our construction in \eqref{eq:hard-f}--\eqref{eq:lower-construction} satisfies all the regularity assumptions (nonconvexity in~$\bx$ and strong-concavity in $\by)$ and is a zero-chain (after subtracting a quadratic regularizer). We can then follow the known rotation techniques in \citet{arjevani2023lower,carmon2020lower} and construct resisting oracles as \citet{nemirovskij1983problem,nesterov2018lectures} to control the progress of information flow and conclude our final lower bound.

\subsection{Final Lower Bound Using Resisting Oracles}

Let the dimension $d_x=d_y=d=3T+1$ and $\bU \in\R^{d\times T}$ be a column-orthogonal matrix. Define
\begin{equation}\label{eq:rotated-instance}
f_\bU(\bx,\by)=-\frac{L}{4} \norm \bx^2+\beta\ip \bx \by
-Q^*( \beta \bU^\top \by)-\frac{\beta^2}{L}\norm{(\bI_d-\bU\bU^\top) \by}^2.
\end{equation}
The above construction decomposes the variable $\by$ into components inside and outside the subspace spanned by $\bU$ as $\by = \bU \bU^\top \by + \by_\perp$ and also decomposes $\bx$ similarly. Since the maximization over~$\by_\perp$ contributes $(L/4)\|\bx_\perp\|^2$, we can reuse Observation \ref{observation-phi} and conclude that the primal objective is the previous one rotated by $\bU$.
\begin{observation} 
We have 
$\Phi_\bU (\bx) \triangleq \max_{\by \in \R^d} f_\bU(\bx,\by) = \alpha \bar h_T(\bU^\top\bx/\sigma)$.
\end{observation}
% Its primal function is
% \begin{equation}\label{eq:rotated-primal}
% \Phi_U(x)=h(U^\top x).
% \end{equation}
% To verify this, split both variables into their components in the range of $U$ and its orthogonal complement. Maximizing the complementary quadratic contributes $\lambda\norm{(I-UU^\top)x}^2/2$, cancelling the complementary $x$ term. The range component gives $h(U^\top x)$. 
We then construct an adversarial rotation matrix $\bU$ such that each query of a deterministic second-order algorithm can only reveal a new hidden direction $\bu_t$. Therefore, the last secret vector~$\bu_T$ remains unknown for the algorithm. 
\begin{lemma} \label{lem:resisting}
For any deterministic second-order oracle as per Definition \ref{dfn:so-alg-class} that makes fewer than~$T$ queries on $f_\bU(\bx,\by)$, the column-orthogonal matrix $\bU$ can be adversarially chosen such that the algorithm output $\bx^{(T)}$ satisfies that $\langle \bx^{(T)}, \bu_T \rangle = 0$.
\end{lemma}
Therefore, we can know from Lemma \ref{lem:carmon}(v) that the output of the algorithm will have a large gradient norm on the primal function $\Phi_\bU(\,\cdot\,)$. This leads to the final lower bound by appropriately setting the parameters as
%The complementary $y$ curvature is $\beta^2/\lambda\in[\mu,5\mu/4]$, so all class-membership bounds proved above remain valid.
% we set and the parameters as 
\begin{equation}\label{eq:scaling}
\rho_h= \frac{\rho\kappa^{\frac{3}{2}}}{10}, \qquad
\sigma=\sqrt{\frac{2\ell_2\eps}{\rho_h}},\qquad
\alpha=\frac{\rho_h \sigma^3}{\ell_2},~~~\text{and}~~~
c_\eps=\ell_1\sqrt{\frac{2\rho_h \eps}{\ell_2}}.
\end{equation}
Under this setting, both conditions in Lemma \ref{lem:pro-psi}--\ref{lem:zero-chain} are of the form $\eps\le cL^2/(\rho\kappa^{3/2})$ with a universal positive constant $c$. Then, we can show that the constructed instance lies in the function class of Assumptions~\ref{ass:smooth}--\ref{ass:gap} and conclude the following lower bound.

\begin{theorem}\label{thm:lower}
There exists a universal constant $c>0$ such that the following statement holds. Let $L\ge16\mu>0$, $\rho>0$, $\Delta>0$, and $0<\eps\le cL^2/(\rho\kappa^{3/2})$.
For every deterministic second-order algorithm, there is a function of the form \eqref{eq:rotated-instance} satisfying Assumptions~\ref{ass:smooth}--\ref{ass:gap}, with $\bz_0=\bzero$, on which the algorithm needs at least
\begin{equation}\label{eq:lower-bound}
T = \Omega \left( \Delta\sqrt\rho\,\kappa^{\frac{3}{4}}\eps^{-\frac{3}{2}}\right)
\end{equation}
queries before it can guarantee $\norm{\nabla\Phi_\bU(\bx^{(t)})}\le\eps$.
%The dimensions are proportional to the chain length in the proof.    
\end{theorem}

% Siqi's work leverages the fact . However, hard to generalize to high order chains. xxx We propose the conjugate embedding xxxx.
% The key is to realize a hard nonconvex function as the primal objective of a smooth minimax problem. A zero-chain hides successive directions: when the unexplored coordinates vanish, a query can reveal at most the next direction in the chain. Its gradient remains large until the last direction is reached. Let $h$ be a scaled smooth zero-chain with Hessian-Lipschitz constant
% \begin{equation}\label{eq:rho-h}
% \rho_h=\theta\rho\kappa^{3/2},
% \end{equation}
% where $\theta>0$ is a numerical constant, and suppose $\norm{\nabla^2h}\le c_\eps\le L/8$. Set

% Adding the quadratic makes $Q$ strongly convex, so the maximizer defining its conjugate is unique. The construction then uses conjugation to place the nonquadratic part of the hard instance in the maximization variable.

% %Moreover,

% Choosing $\theta$ small enough ensures that the last bound is at most $\rho$. The full Hessian of $f$ also has norm at most $L$. Thus, the primal can have third-derivative scale $\rho\kappa^{3/2}$ while the original minimax objective has third-derivative scale only $\rho$.
% numerical part
\section{Numerical Experiments}\label{sec:experiments}
We test ECRN on robust logistic regression with the objective function defined by
\[
f(\bx,\by)=\frac{1}{n}\sum_{i=1}^n\log\bigl(1+\exp(-b_i\ba_i^\top(\bx+\by))\bigr)+r(\bx)-\frac{\nu}{2}\norm{\by}^2,
\]
where $\{\ba_i,b_i\}_{i=1}^n$ are the training sets and $r(\cdot)$ is a smooth nonconvex regularizer.
We compare ECRN with MCN method~\citep{luo2022}, TT-GDA~\citep{lin2020gda}, and IAPUN~\citep{wang2024} on the datasets ``a9a'', ``ijcnn'', and ``w8a'' in LIBSVM~\citep{chang2011libsvm}.
We report the gradient norm versus the number of second-order oracles and CPU time with different condition number $\kappa = 1000$ and $\kappa=100$ in Figure~\ref{fig:experiments} and~\ref{fig:exp-k100}, respectively.
The experimental details are deferred to Appedix~\ref{app:experiments}.
\begin{figure}[p]
\centering
\includegraphics[width=0.94\textwidth]{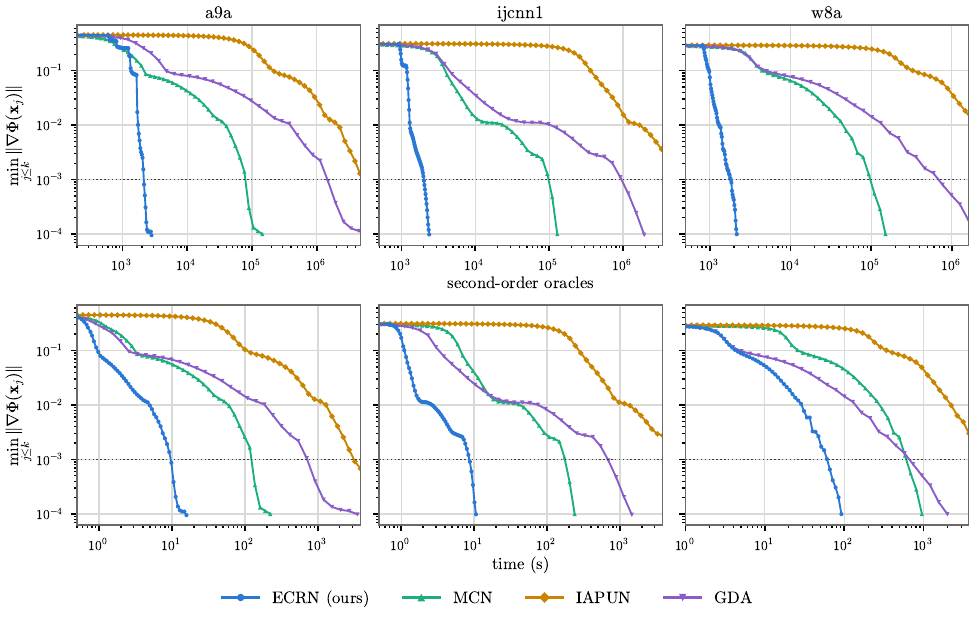}
\caption{Robust logistic regression with $\kappa=1000$.}
\label{fig:experiments}
\vspace{10pt}
\includegraphics[width=0.94\textwidth]{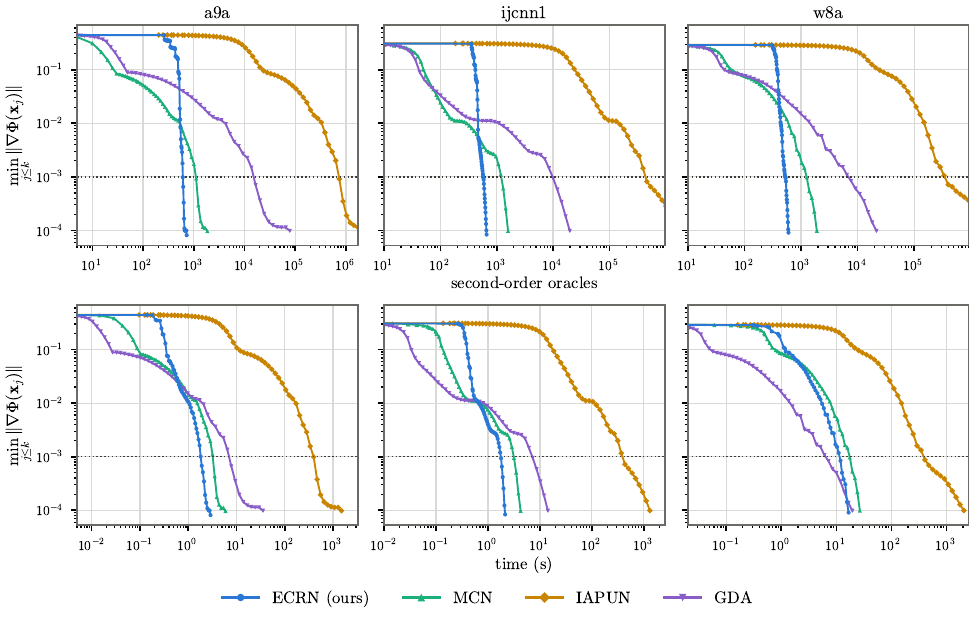}
\caption{Robust logistic regression with $\kappa=100$.}
\label{fig:exp-k100}
\end{figure}

\section{Closing Remarks}
We discuss several directions for improving ECRN and extending its analysis to other settings.

\paragraph{Efficient Variants of ECRN.}
ECRN requires exact Hessian evaluations of $f(\cdot,\cdot)$. Reusing Hessians and their factorizations may reduce its computational cost, as in lazy Hessian methods~\citep{doikov2023lazy,liu2025lazy,yang2026bilevel}.
Its nested structure also motivates a single-loop variant, building on recent advances in minimax optimization~\citep{zhang2026matching,ma2026} and single-loop cubic regularization for bilevel optimization~\citep{bi2026secondorder}.
Finally, adaptive selection of the regularization and inner accuracy may lead to a parameter-free variant that does not require prior knowledge of problem parameters~\citep{wang2026parameterfree}.

\paragraph{Extension to Higher-Order Methods.} 
The prior results mainly focus on higher-order methods for convex-concave minimax problem~~\citep{bullins2022higher,adil2022optimal,huang2022approximation,jiang2022generalized,lin2022perseus,chen2026solving,zhang2026optimal}.
It will be interesting to extend the results to higher-order methods for solving NC-SC minimax problem.
The $p$-th order methods for nonconvex minimization problem achieve $\mathcal{O}(\epsilon^{-(p+1)/p})$ complexity~\cite{birgin2017highorder,cartis2020arbitrary}. 
It is possible to study the higher-order Lipschitz parameter of the primal envelop to achieve a sharp upper bound.

\paragraph{Acceleration of First-Order Methods.}
The sharper Hessian-Lipschitz bound of $\Phi_{2L}(\cdot)$ can also improve the condition-number dependence of accelerated first-order methods.
The envelope has a $2L$-Lipschitz gradient and a $\widetilde\rho$-Lipschitz Hessian. 
For $\eps\le cL^2/\widetilde\rho$, restarted AGD~\citep[Theorem~1]{li2023restarted} therefore requires
$
\cO\bigl(\Delta L^{1/2}\widetilde\rho^{1/4}\eps^{-7/4}\bigr)
=\cO\bigl(\Delta L^{1/2}\rho^{1/4}\kappa^{3/8}\eps^{-7/4}\bigr)$
outer iterations to find an approximate first-order stationary point of the envelope.
The inexact-gradient analysis of \citet[Theorem~13]{yang2023accelerating} permits polynomially small gradient errors, while accelerated first-order methods solve each regularized saddle problem~\eqref{obj:minimax} to the required accuracy using $\tO(\sqrt\kappa)$ gradient evaluations~\citep{lin2020near,kovalev2022}.
Thus, it is possible to obtain the total complexity bound of 
$
\tO\bigl(\Delta L^{1/2}\rho^{1/4}\kappa^{7/8}\eps^{-7/4}\bigr),
$
which can improve the condition-number dependence over the $\tO(\kappa^{7/4}\eps^{-7/4})$ bound of PRAGDA~\citep{yang2023accelerating} and the $\tO(\kappa^{5/4}\eps^{-7/4})$ bound of IAPUN~\citep{wang2024} for first-order stationarity.
Establishing a matching first-order lower bound under Hessian smoothness also remains an open direction.

\paragraph{Reducing the Condition Number of Nested Methods.}
NewtonTrack may also improve the inner solves of other second-order methods for minimax~\citep{luo2022} and bilevel optimization problems~\citep{yang2026bilevel}.
Replacing gradient-based subsolvers with NewtonTrack could reduce their condition-number dependence when solution sensitivity and local convergence can be controlled along the outer iterates.
\bibliographystyle{plainnat}
\bibliography{reference}
\appendix
\paragraph{AI usage.}
The problem of reducing the condition number dependency of second-order methods~\cite{luo2022} for nonconvex-strongly-concave (NC-SC) minimax optimization was discussed among the authors since 2022.
The authors came up with the idea of envelope smoothing, then obtained an upper bound of $\tilde{\mathcal{O}}(\kappa^{3/4+{2/7}}\epsilon^{-3/2})$ by applying inexact CRN on $\Phi_{2L}(\cdot)$ with the accelerated second-order method for unbalanced SC-SC problem~\cite{chen2025solving}.
The authors found sharp upper and lower bounds presented in this paper after several rounds of discussion with ChatGPT 6 Astra. 
The results were checked, re-organized, and simplified by the authors.

\section{The Proof of Section~\ref{sec:prelim}}\label{app:prelim}

\subsection{Proof of Lemma~\ref{lem:cubic}}
\begin{proof}
We specialize \citet[Lemmas~2--3]{wang2019inexact} to the cubic parameter $M=2\rho_g$ and the errors in~\eqref{eq:model-contract}. Writing $r_k=\norm{\bs_k}$, their Taylor and cubic-model estimates give
\begin{equation}\label{eq:inexact-decrease}
\begin{aligned}
&g(\bz_k)-g(\bz_{k+1})
\ge\frac{\rho_g}{3}r_k^3-\frac{\eps}{128}r_k
-\frac{\sqrt{\rho_g\eps}}{256}r_k^2,\\
&\norm{\nabla g(\bz_{k+1})}
\le\frac{3\rho_g}{2}r_k^2+\frac{\sqrt{\rho_g\eps}}{128}r_k+\frac{\eps}{128},\\
\text{and}\qquad&\nabla^2g(\bz_{k+1})
\succeq-\Bigl(2\rho_g r_k+\frac{\sqrt{\rho_g\eps}}{128}\Bigr)\bI_d.
\end{aligned}
\end{equation}
Whenever $r_k>\sqrt{\eps/(8\rho_g)}$, we have
\begin{equation}\label{eq:cubic-decrease}
\begin{aligned}
g(\bz_k)-g(\bz_{k+1})
&\stackrel{\eqref{eq:inexact-decrease}}{\ge}
\frac{\rho_g}{3}r_k^3-\frac{\eps}{128}r_k-\frac{\sqrt{\rho_g\eps}}{256}r_k^2\\
&\quad {\ge}
\rho_g\Bigl(\frac{1}{3}-\frac{8}{128}-\frac{\sqrt8}{256}\Bigr)r_k^3
\ge\frac{\rho_g}{4}r_k^3.
\end{aligned}
\end{equation}
If the stopping condition is not satisfied at $k=0,\ldots,N-1$, then
\begin{equation}\label{eq:cubic-telescope}
\begin{aligned}
\frac{N\eps^{\frac{3}{2}}}{4\cdot8^{\frac{3}{2}}\sqrt{\rho_g}}
&\le\frac{\rho_g}{4}\sum_{k=0}^{N-1}r_k^3
\stackrel{\eqref{eq:cubic-decrease}}{\le}
\sum_{k=0}^{N-1}\bigl(g(\bz_k)-g(\bz_{k+1})\bigr)\\
&=g(\bz_0)-g(\bz_N)\le\Delta.
\end{aligned}
\end{equation}
Thus the stopping condition can fail at most $4\cdot8^{3/2}\Delta\sqrt{\rho_g}\,\eps^{-3/2}$ times. Let $K$ be the first index such that $r_K\le\sqrt{\eps/(8\rho_g)}$. Taking $N=K$ in~\eqref{eq:cubic-telescope} proves both bounds in~\eqref{eq:cubic-sum}. At iteration $K$, we have
\[
\begin{aligned}
&\norm{\nabla g(\bz_{K+1})}
\stackrel{\eqref{eq:inexact-decrease}}{\le}
\frac{3\rho_g}{2}r_K^2+\frac{\sqrt{\rho_g\eps}}{128}r_K+\frac{\eps}{128}
\le\Bigl(\frac{3}{16}+\frac{1}{128\sqrt8}+\frac{1}{128}\Bigr)\eps<\frac{\eps}{2},\\
\text{and}\qquad&\nabla^2g(\bz_{K+1})
\stackrel{\eqref{eq:inexact-decrease}}{\succeq}
-\Bigl(2\rho_g r_K+\frac{\sqrt{\rho_g\eps}}{128}\Bigr)\bI_d\succeq-\Bigl(\frac{1}{\sqrt2}+\frac{1}{128}\Bigr)\sqrt{\rho_g\eps}\,\bI_d
\succeq-\frac{3}{4}\sqrt{\rho_g\eps}\,\bI_d.
\end{aligned}
\]
This proves~\eqref{eq:cubic-stop}.
\end{proof}

\section{The Proof of Section~\ref{sec:envelope-saddle}}\label{app:envelope-saddle}

\subsection{Proof of Proposition~\ref{prop:primal}}
\begin{proof}
\emph{(i)} The Hessian blocks satisfy
\begin{equation}\label{eq:regularized-block-bounds}
\begin{aligned}
&\nabla^2_{uu}f_{\bz}(\bw)
\stackrel{\eqref{eq:regularized-derivatives}}{=}
\nabla^2_{uu}f(\bw)+2L\bI_{d_x}
\stackrel{\eqref{eq:smooth}}{\succeq}L\bI_{d_x}~~~
\text{and}~~~\nabla^2_{yy}f_{\bz}(\bw)
\stackrel{\eqref{eq:regularized-derivatives}}{=}
\nabla^2_{yy}f(\bw)
{\preceq}-\mu\bI_{d_y}.
\end{aligned}
\end{equation}
Thus $f_{\bz}(\cdot)$ is $L$-strongly convex in $\bu$ and $\mu$-strongly concave in $\by$. 
This gives $\norm{\bigl(\nabla^2 f_{\bz}(\bw)\bigr)^{-1}}\le \mu^{-1}$.
Its unique saddle point~\eqref{eq:saddle-map} exists by \citet[Lemma~1.1 and Theorem~1.2]{grimmer2023landscape}, applied globally to $f_{\bz}(\cdot)$.
% For every $\ba\in\R^{d_x}$ and $\bb\in\R^{d_y}$, we have
% \begin{equation}\label{eq:regularized-hessian-coercivity}
% \begin{aligned}
% &\mu\bigl(\norm{\ba}^2+\norm{\bb}^2\bigr)
% \le L\norm{\ba}^2+\mu\norm{\bb}^2
% \stackrel{\eqref{eq:regularized-block-bounds}}{\le}
% \begin{pmatrix}\ba\\-\bb\end{pmatrix}^{\!\top}
% \nabla^2f_{\bz}(\bw)\begin{pmatrix}\ba\\\bb\end{pmatrix}\le\norm{\begin{pmatrix}\ba\\\bb\end{pmatrix}}
% \norm{\nabla^2f_{\bz}(\bw)\begin{pmatrix}\ba\\\bb\end{pmatrix}}.
% \end{aligned}
% \end{equation}
% Hence the smallest singular value of $\nabla^2f_{\bz}(\bw)$ is at least $\mu$, proving~\eqref{eq:regularized-inverse-bound}.

\emph{(ii)} We have
\[
\begin{aligned}
&\norm{\nabla^2f_{\bz}(\bw)}
\stackrel{\eqref{eq:regularized-derivatives}}{\le}
\norm{\nabla^2f(\bw)}+2L
\stackrel{\eqref{eq:smooth}}{\le}3L\\
\text{and}\qquad&\norm{\nabla^2f_{\bz}(\bw)-\nabla^2f_{\bz}(\bv)}
\stackrel{\eqref{eq:regularized-derivatives}}{=}
\norm{\nabla^2f(\bw)-\nabla^2f(\bv)}
\stackrel{\eqref{eq:smooth}}{\le}\rho\norm{\bw-\bv}.
\end{aligned}
\]

\end{proof}

\subsection{Proof of Proposition~\ref{prop:envelope-basic}}
\begin{proof}
Equation~(10) in \citet{grimmer2023landscape}, with proximal parameter $2L$, gives the gradient identity~\eqref{eq:envelope-basic}.
By~\eqref{eq:regularized-inverse-bound}, the implicit function theorem applied to $\nabla f_{\bz}(\bw^*(\bz))=\bzero_{d_x+d_y}$ gives $\bw^*(\cdot)\in C^1$ with
\begin{equation}\label{eq:saddle-jacobian}
\begin{aligned}
\frac{\partial\bw^*(\bz)}{\partial\bz}
&=\bigl(\nabla^2f_{\bz}(\bw^*(\bz))\bigr)^{-1}
\begin{pmatrix}2L\bI_{d_x}\\\bzero_{d_y\times d_x}\end{pmatrix}\\
&\!\!\!\!\!\stackrel{\eqref{eq:regularized-derivatives},\eqref{eq:schur-complement}}{=}
2L\begin{pmatrix}
\bM(\bw^*(\bz))^{-1}\\
-\bigl(\nabla^2_{yy}f(\bw^*(\bz))\bigr)^{-1}\nabla^2_{yu}f(\bw^*(\bz))\bM(\bw^*(\bz))^{-1}
\end{pmatrix}.
\end{aligned}
\end{equation}
The Schur complement satisfies
\[
\bM(\bw)
\stackrel{\eqref{eq:schur-complement},\eqref{eq:regularized-block-bounds}}{\succeq}
\nabla^2_{uu}f(\bw)+2L\bI_{d_x}
\stackrel{\eqref{eq:regularized-block-bounds}}{\succeq}L\bI_{d_x}.
\]
In addition, \eqref{eq:saddle-jacobian} gives
\begin{equation}\label{eq:sensitivity}
\frac{\partial\bu^*(\bz)}{\partial\bz}
{=}
2L\bM(\bw^*(\bz))^{-1}.
\end{equation}
Differentiating the gradient identity gives
\[
\nabla^2\Phi_{2L}(\bz)
\stackrel{\eqref{eq:envelope-basic}}{=}
2L\bI_{d_x}-2L\frac{\partial\bu^*(\bz)}{\partial\bz}
\stackrel{\eqref{eq:sensitivity}}{=}
2L\bI_{d_x}-4L^2\bM(\bw^*(\bz))^{-1},
\]
which proves~\eqref{eq:envelope-hessian} and $\Phi_{2L}(\cdot)\in C^2$.

\emph{(i)} Equation~(9) in \citet{grimmer2023landscape} gives
\begin{equation}\label{eq:envelope-value-upper}
\Phi_{2L}(\bz)\le\Phi(\bz).
\end{equation}
Moreover,
\begin{equation}\label{eq:envelope-value-bounds}
\inf_{\bz}\Phi_{2L}(\bz)
\stackrel{\eqref{eq:envelope}}{=}
\inf_{\bu}\Bigl\{\Phi(\bu)+\inf_{\bz}L\norm{\bu-\bz}^2\Bigr\}
=\inf_{\bu}\Phi(\bu).
\end{equation}
Hence
\[
\Phi_{2L}(\bz_0)-\inf_{\bz}\Phi_{2L}(\bz)
\stackrel{\eqref{eq:envelope-value-upper},\eqref{eq:envelope-value-bounds}}{\le}
\Phi(\bz_0)-\inf_{\bz}\Phi(\bz) {\le}\Delta.
\]

\emph{(ii)} Proposition~1 in \citet{chen2023cubic} gives
\begin{equation}\label{eq:primal-schur-hessian}
\nabla^2\Phi(\bu^*(\bz))=\bM(\bw^*(\bz))-2L\bI_{d_x}.
\end{equation}
Suppose $\norm{\nabla\Phi_{2L}(\bz)}\le\eps$ and $\nabla^2\Phi_{2L}(\bz)\succeq-\eta\bI_{d_x}$. Then
\begin{equation}\label{eq:envelope-inverse-curvature}
4L^2\bM(\bw^*(\bz))^{-1}
\stackrel{\eqref{eq:envelope-hessian}}{=}
2L\bI_{d_x}-\nabla^2\Phi_{2L}(\bz)
\preceq(2L+\eta)\bI_{d_x}.
\end{equation}
Since $\bM(\bw^*(\bz))\succeq L\bI_{d_x}$, we obtain
\[
\begin{aligned}
&\nabla^2\Phi(\bu^*(\bz))
\stackrel{\eqref{eq:primal-schur-hessian}}{=}
\bM(\bw^*(\bz))-2L\bI_{d_x}
\stackrel{\eqref{eq:envelope-inverse-curvature}}{\succeq}
\Bigl(\frac{4L^2}{2L+\eta}-2L\Bigr)\bI_{d_x}=-\frac{2L\eta}{2L+\eta}\bI_{d_x}\succeq-\eta\bI_{d_x}\\
\end{aligned}
\]
and
\begin{align*}
    \norm{\nabla\Phi(\bu^*(\bz))}
\stackrel{\eqref{eq:envelope-basic}}{=}\norm{\nabla\Phi_{2L}(\bz)}\le\eps.
\end{align*}
\end{proof}

\subsection{Proof of Lemma~\ref{lem:saddle-map}}
\begin{proof}
As shown in the proof of Proposition~\ref{prop:envelope-basic}, $\bw^*(\cdot)\in C^1$ and~\eqref{eq:saddle-jacobian} holds.
For every $\bw\in\R^{d_x+d_y}$, we have
\begin{equation}\label{eq:inverse-upper-block}
\begin{aligned}
 \norm{2L\bM(\bw)^{-1}}^2
\le 4L^2\cdot L^{-2}=4,
\end{aligned}
\end{equation}
where we use $\bM(\bw)\succeq L\bI_{d_x}$.
For the lower block, we have
\begin{equation}\label{eq:schur-block-bounds}
\begin{aligned}
\nabla^2_{uy}f(\bw)\bigl(\nabla^2_{yy}f(\bw)\bigr)^{-2}\nabla^2_{yu}f(\bw) &\preceq
-\mu^{-1}\nabla^2_{uy}f(\bw)\bigl(\nabla^2_{yy}f(\bw)\bigr)^{-1}\nabla^2_{yu}f(\bw)\\
&\stackrel{\eqref{eq:schur-complement}}{=}
\mu^{-1}\bigl(\bM(\bw)-\nabla^2_{uu}f(\bw)-2L\bI_{d_x}\bigr)\\
&\stackrel{\eqref{eq:regularized-block-bounds}}{\preceq}
\mu^{-1}\bigl(\bM(\bw)-L\bI_{d_x}\bigr).
\end{aligned}
\end{equation}
Moreover, we have
\begin{equation}\label{eq:schur-inverse-quadratic-bound}
\begin{aligned}
&\bzero_{d_x\times d_x}
{\preceq}
\bM(\bw)^{-1}-L\bM(\bw)^{-2}=\frac{1}{4L}\bI_{d_x}
-L\Bigl(\bM(\bw)^{-1}-\frac{1}{2L}\bI_{d_x}\Bigr)^2
\preceq\frac{1}{4L}\bI_{d_x}.
\end{aligned}
\end{equation}
Thus the squared norm of the lower block satisfies
\begin{equation}\label{eq:inverse-lower-block}
\begin{aligned}
&\norm{2L\bigl(\nabla^2_{yy}f(\bw)\bigr)^{-1}\nabla^2_{yu}f(\bw)\bM(\bw)^{-1}}^2\\
&=4L^2\norm{\bM(\bw)^{-1}\nabla^2_{uy}f(\bw)
\bigl(\nabla^2_{yy}f(\bw)\bigr)^{-2}\nabla^2_{yu}f(\bw)\bM(\bw)^{-1}}\\
&\!\!\stackrel{\eqref{eq:schur-block-bounds}}{\le}
\frac{4L^2}{\mu}\norm{\bM(\bw)^{-1}
\bigl(\bM(\bw)-L\bI_{d_x}\bigr)\bM(\bw)^{-1}}\\
&=\frac{4L^2}{\mu}\norm{\bM(\bw)^{-1}-L\bM(\bw)^{-2}}
\stackrel{\eqref{eq:schur-inverse-quadratic-bound}}{\le}
\frac{4L^2}{\mu}\cdot\frac{1}{4L}=\kappa.
\end{aligned}
\end{equation}
Combining the two block estimates gives
\begin{equation}\label{eq:T-bound}
\begin{aligned}
&\norm{\bigl(\nabla^2f_{\bz}(\bw)\bigr)^{-1}
\begin{pmatrix}2L\bI_{d_x}\\\bzero_{d_y\times d_x}\end{pmatrix}}^2\\
&\quad\stackrel{\eqref{eq:regularized-derivatives},\eqref{eq:schur-complement}}{\le}
\norm{2L\bM(\bw)^{-1}}^2
+\norm{2L\bigl(\nabla^2_{yy}f(\bw)\bigr)^{-1}\nabla^2_{yu}f(\bw)\bM(\bw)^{-1}}^2\\
&\quad\stackrel{\eqref{eq:inverse-upper-block},\eqref{eq:inverse-lower-block}}{\le}4+\kappa.
\end{aligned}
\end{equation}
Integrating the derivative along the segment from $\bz'$ to $\bz$ yields
\[
\begin{aligned}
&\norm{\bw^*(\bz)-\bw^*(\bz')}
=\norm{\int_0^1\frac{\partial\bw^*}{\partial\bz}\bigl(\bz'+t(\bz-\bz')\bigr)(\bz-\bz')\,dt}\stackrel{\eqref{eq:saddle-jacobian},\eqref{eq:T-bound}}{\le}
\sqrt{4+\kappa}\norm{\bz-\bz'}.
\end{aligned}
\]
Finally, the distance to the saddle point satisfies
\[
\norm{\bw-\bw^*(\bz)} {\le}
\mu^{-1}\norm{\nabla f_{\bz}(\bw) -\nabla f_{\bz}(\bw^*(\bz))} = \mu^{-1} \|\nabla f_{\bz}(\bw)\|,
\]
which proves~\eqref{eq:residual-to-distance}.
\end{proof}

\subsection{Proof of Lemma~\ref{lem:pce-dm}}
\begin{proof}
For $\bw_0=\bw_{\rm init}$, equation~(24) in \citet{zhang2026near} gives
\begin{equation}\label{eq:pce-rate}
\norm{\bw_t-\bw^*(\bz)}^2
\le
288\kappa^{\frac{3}{2}}\Bigl(1+\frac{1}{360\sqrt\kappa}\Bigr)^{-t}D_{\rm init}^2.
\end{equation}
Since $\kappa\ge1$, we have
\begin{equation}\label{eq:pce-log-bound}
\log\Bigl(1+\frac{1}{360\sqrt\kappa}\Bigr)
\ge\frac{1}{360\sqrt\kappa+1}
\ge\frac{1}{720\sqrt\kappa}.
\end{equation}
Since $\nabla f_{\bz}(\bw^*(\bz))=\bzero_{d_x+d_y}$, we have
\begin{equation}\label{eq:pce-gradient-rate}
\begin{aligned}
&\norm{\nabla f_{\bz}(\bw_t)}^2
\stackrel{\eqref{eq:regularized-hessian-bounds}}{\le}
9L^2\norm{\bw_t-\bw^*(\bz)}^2\\
&\quad\stackrel{\eqref{eq:pce-rate}}{\le}
2592L^2\kappa^{\frac{3}{2}}D_{\rm init}^2
\Bigl(1+\frac{1}{360\sqrt\kappa}\Bigr)^{-t}\\
&\quad\stackrel{\eqref{eq:pce-log-bound}}{\le}
2592L^2\kappa^{\frac{3}{2}}D_{\rm init}^2
\exp\Bigl(-\frac{t}{720\sqrt\kappa}\Bigr).
\end{aligned}
\end{equation}
Taking $t=T_{\rm PCE}(\bar\eps)$ gives
\[
\begin{aligned}
&\norm{\nabla f_{\bz}(\bw_{T_{\rm PCE}(\bar\eps)})}^2\\
&\quad\stackrel{\eqref{eq:pce-gradient-rate}}{\le}
2592L^2\kappa^{\frac{3}{2}}D_{\rm init}^2
\exp\Bigl(-\frac{T_{\rm PCE}(\bar\eps)}{720\sqrt\kappa}\Bigr)\\
&\quad\stackrel{\eqref{eq:pce-horizon}}{\le}
2592L^2\kappa^{\frac{3}{2}}D_{\rm init}^2
\exp\Bigl(-\log\frac{2592L^2\kappa^{\frac{3}{2}}D_{\rm init}^2}{\bar\eps^2}\Bigr)
=\bar\eps^2.
\end{aligned}
\]
Each iteration evaluates the gradient at the predictor and at the new iterate, reusing the latter for the next stopping test and predictor. Including the initial query, the gradient count is at most
\[
2T_{\rm PCE}(\bar\eps)+1
\stackrel{\eqref{eq:pce-horizon}}{=}\tO(\sqrt\kappa).
\]
\end{proof}

\subsection{Proof of Lemma~\ref{lem:local-newton}}
\begin{proof}
Fix $\bz$ and let $\bw_t$ be the Newton iterates, with $\bw_0=\bw_{\rm init}$. Since $\nabla f_{\bz}(\bw^*(\bz))=\bzero_{d_x+d_y}$, Taylor's formula gives
\begin{equation}\label{eq:newton-taylor-remainder}
\norm{\nabla f_{\bz}(\bw_t)-\nabla^2f_{\bz}(\bw_t)(\bw_t-\bw^*(\bz))}
\stackrel{\eqref{eq:regularized-hessian-bounds}}{\le}
\frac{\rho}{2}\norm{\bw_t-\bw^*(\bz)}^2.
\end{equation}
The Newton update in Algorithm~\ref{alg:newton} therefore satisfies
\begin{equation}\label{eq:newton-error-recursion}
\begin{aligned}
&\norm{\bw_{t+1}-\bw^*(\bz)}\\
&\quad\stackrel{\eqref{eq:regularized-inverse-bound}}{\le}
\frac{1}{\mu}\norm{\nabla f_{\bz}(\bw_t)-\nabla^2f_{\bz}(\bw_t)(\bw_t-\bw^*(\bz))}\\
&\quad\stackrel{\eqref{eq:newton-taylor-remainder}}{\le}
\frac{\rho}{2\mu}\norm{\bw_t-\bw^*(\bz)}^2.
\end{aligned}
\end{equation}
For every $t\ge0$, it follows that
\begin{equation}\label{eq:newton-distance}
\begin{aligned}
&\norm{\bw_t-\bw^*(\bz)}
\stackrel{\eqref{eq:newton-error-recursion}}{\le}
\frac{2\mu}{\rho}\Bigl(\frac{\rho\norm{\bw_0-\bw^*(\bz)}}{2\mu}\Bigr)^{2^t}\stackrel{\eqref{eq:local-newton-condition}}{\le}
\frac{2\mu}{\rho}\,2^{-2^t}.
\end{aligned}
\end{equation}
Consequently, we have
\begin{equation}\label{eq:newton-residual-rate}
\norm{\nabla f_{\bz}(\bw_t)}
\stackrel{\eqref{eq:regularized-hessian-bounds}}{\le} 
3L\norm{\bw_t-\bw^*(\bz)}
\stackrel{\eqref{eq:newton-distance}}{\le}
\frac{6L\mu}{\rho}\,2^{-2^t}.
\end{equation}
Taking $t=T_{\rm N}(\bar\eps)$ in the gradient bound gives
\[
\norm{\nabla f_{\bz}(\bw_{T_{\rm N}(\bar\eps)})}
\stackrel{\eqref{eq:newton-residual-rate}}{\le}
\frac{6L\mu}{\rho}\,2^{-2^{T_{\rm N}(\bar\eps)}}
\stackrel{\eqref{eq:local-newton-horizon}}{\le}\bar\eps.
\]
Algorithm~\ref{alg:newton} therefore stops after at most $T_{\rm N}(\bar\eps)$ updates. 
\end{proof}

\section{The Proof of Section~\ref{sec:algorithm}}\label{app:algorithm}

\subsection{Proof of Lemma~\ref{lem:oracle}}
\begin{proof}
We assume that 
\begin{equation}\label{eq:oracle-residual-assumption}
\norm{\nabla f_{\bz}(\bw)}\le\widetilde\eps.
\end{equation}
Hence, we have
\begin{equation}\label{eq:oracle-distance}
\norm{\bu-\bu^*(\bz)}
\le\norm{\bw-\bw^*(\bz)}
\stackrel{\eqref{eq:residual-to-distance}}{\le}
\frac{\norm{\nabla f_{\bz}(\bw)}}{\mu}
\stackrel{\eqref{eq:oracle-residual-assumption}}{\le}
\frac{\widetilde\eps}{\mu}.
\end{equation}
The gradient approximation satisfies
\[
\begin{aligned}
\norm{\widehat{\bg}-\nabla\Phi_{2L}(\bz)}
\stackrel{\eqref{eq:approx-derivatives},\eqref{eq:envelope-basic}}{=}
2L\norm{\bu-\bu^*(\bz)}\stackrel{\eqref{eq:oracle-distance}}{\le}
\frac{2L\widetilde\eps}{\mu}
\stackrel{\eqref{eq:parameters}}{\le}\frac{\eps}{128}.
\end{aligned}
\]
The Schur complement $\bM(\bw)$ is symmetric, so $\widehat{\bH}$ is symmetric by~\eqref{eq:approx-derivatives}. Its approximation error satisfies
\[
\begin{aligned}
\norm{\widehat{\bH}-\nabla^2\Phi_{2L}(\bz)}
&\stackrel{\eqref{eq:approx-derivatives},\eqref{eq:envelope-hessian}}{=}
4L^2\norm{\bM(\bw^*(\bz))^{-1}-\bM(\bw)^{-1}}\\
&\quad\stackrel{\eqref{eq:schur-inverse-lipschitz}}{\le}
\rho(4+\kappa)\norm{\bw-\bw^*(\bz)}\\
&\quad\stackrel{\eqref{eq:oracle-distance}}{\le}
\frac{\rho(4+\kappa)\widetilde\eps}{\mu}
\stackrel{\eqref{eq:parameters}}{\le}
\frac{\sqrt{\widetilde\rho\eps}}{128},
\end{aligned}
\]
which proves~\eqref{eq:model-errors}.
We let $\bar\by:=\argmax_{\by}f(\bu,\by)$ such that $\nabla_y f(\bu,\bar\by)=\bzero_{d_y}$, then we have
\begin{equation}\label{eq:oracle-dual-monotonicity}
\begin{aligned}
&\mu\norm{\by-\bar\by}^2
{\le}
-\ip{\nabla_y f(\bu,\by)}{\by-\bar\by} \le\norm{\nabla_y f(\bu,\by)}\norm{\by-\bar\by}.
\end{aligned}
\end{equation}
Thus, the dual error obeys
\begin{equation}\label{eq:oracle-dual-distance}
\begin{aligned}
&\norm{\by-\bar\by}
\stackrel{\eqref{eq:oracle-dual-monotonicity}}{\le}
\frac{\norm{\nabla_y f(\bu,\by)}}{\mu}\stackrel{\eqref{eq:regularized-derivatives}}{=}
\frac{\norm{\nabla_y f_{\bz}(\bw)}}{\mu}
\stackrel{\eqref{eq:oracle-residual-assumption}}{\le}
\frac{\widetilde\eps}{\mu}.
\end{aligned}
\end{equation}
Danskin's identity gives
\begin{equation}\label{eq:oracle-primal-gradient-decomposition}
\begin{aligned}
&\nabla\Phi(\bu)-2L(\bz-\bu)\stackrel{\eqref{eq:regularized-derivatives}}{=}
\nabla_u f(\bu,\bar\by)-\nabla_u f(\bu,\by)
+\nabla_u f_{\bz}(\bw).
\end{aligned}
\end{equation}
Thus, we have
\begin{equation}\label{eq:primal-gradient-error}
\begin{aligned}
&\norm{\nabla\Phi(\bu)-\nabla\Phi_{2L}(\bz)}\\
&\quad\stackrel{\eqref{eq:envelope-basic}}{\le}
\norm{\nabla\Phi(\bu)-2L(\bz-\bu)}+2L\norm{\bu-\bu^*(\bz)}\\
&\quad\stackrel{\eqref{eq:oracle-primal-gradient-decomposition}}{\le}
\norm{\nabla_u f(\bu,\bar\by)-\nabla_u f(\bu,\by)}
+\norm{\nabla_u f_{\bz}(\bw)}+2L\norm{\bu-\bu^*(\bz)}\\
&\quad\stackrel{\eqref{eq:smooth}}{\le}
L\norm{\by-\bar\by}+\norm{\nabla_u f_{\bz}(\bw)}+2L\norm{\bu-\bu^*(\bz)}\\
&\!\!\!\!\!\!\!\!\!\quad\stackrel{\eqref{eq:oracle-dual-distance},\eqref{eq:oracle-residual-assumption},\eqref{eq:oracle-distance}}{\le}
\frac{L\widetilde\eps}{\mu}+\widetilde\eps+\frac{2L\widetilde\eps}{\mu}
\le\frac{4L\widetilde\eps}{\mu}.
\end{aligned}
\end{equation}
The Hessian formulas give
\begin{equation}\label{eq:oracle-exact-primal-curvature}
\begin{aligned}
&\nabla^2\Phi(\bu^*(\bz))-\nabla^2\Phi_{2L}(\bz)\\
&\quad\stackrel{\eqref{eq:primal-schur-hessian},\eqref{eq:envelope-hessian}}{=}
\bigl(\bM(\bw^*(\bz))-2L\bI_{d_x}\bigr)
\bM(\bw^*(\bz))^{-1}
\bigl(\bM(\bw^*(\bz))-2L\bI_{d_x}\bigr)
\succeq\bzero_{d_x\times d_x}.
\end{aligned}
\end{equation}
Proposition~1 in \citet{chen2023cubic} gives
\begin{equation}\label{eq:oracle-primal-hessian-lipschitz}
\norm{\nabla^2\Phi(\bu)-\nabla^2\Phi(\bu^*(\bz))}{\le}
\rho(1+\kappa)^3\norm{\bu-\bu^*(\bz)}.
\end{equation}
Therefore, we have
\begin{equation}\label{eq:output-error}
\begin{aligned}
&\nabla^2\Phi(\bu)\\
&\stackrel{\eqref{eq:oracle-primal-hessian-lipschitz}}{\succeq}
\nabla^2\Phi(\bu^*(\bz))
-\rho(1+\kappa)^3\norm{\bu-\bu^*(\bz)}\,\bI_{d_x}\\
&\!\!\!\stackrel{\eqref{eq:oracle-exact-primal-curvature},\eqref{eq:oracle-distance}}{\succeq}
\nabla^2\Phi_{2L}(\bz)
-\frac{\rho(1+\kappa)^3\widetilde\eps}{\mu}\,\bI_{d_x}.
\end{aligned}
\end{equation}
If $\bz$ is an $(\epsilon/2, 3\sqrt{\widetilde{\rho}\epsilon}/4)$-SOSP such that 
\begin{equation}\label{eq:oracle-envelope-stationarity}
\begin{aligned}
\norm{\nabla\Phi_{2L}(\bz)}\le\frac{\eps}{2}~~~
\text{and}~~~\nabla^2\Phi_{2L}(\bz)
\succeq-\frac{3}{4}\sqrt{\widetilde\rho\eps}\,\bI_{d_x}.
\end{aligned}
\end{equation}
Then, for the computed point $\bw= (\bu,\by)$ satisfied \eqref{eq:oracle-residual-assumption}, we have
\begin{align*}
&\norm{\nabla\Phi(\bu)}\\
&\quad\le\norm{\nabla\Phi_{2L}(\bz)}
+\norm{\nabla\Phi(\bu)-\nabla\Phi_{2L}(\bz)}\\
&\stackrel{\eqref{eq:oracle-envelope-stationarity},\eqref{eq:primal-gradient-error}}{\le}
\frac{\eps}{2}+\frac{4L\widetilde\eps}{\mu}
\stackrel{\eqref{eq:parameters}}{\le}\frac{\eps}{2}+\frac{\eps}{64}<\eps
\end{align*}
and 
\begin{align*}
&\nabla^2\Phi(\bu)\\
&\quad\stackrel{\eqref{eq:output-error}}{\succeq}
\nabla^2\Phi_{2L}(\bz)
-\frac{\rho(1+\kappa)^3\widetilde\eps}{\mu}\,\bI_{d_x}\\
&\quad\stackrel{\eqref{eq:oracle-envelope-stationarity}}{\succeq}
-\Bigl(\frac{3}{4}\sqrt{\widetilde\rho\eps}
+\frac{\rho(1+\kappa)^3\widetilde\eps}{\mu}\Bigr)\bI_{d_x}\\
&\quad\stackrel{\eqref{eq:parameters}}{\succeq}
-\Bigl(\frac{3}{4}+\frac{1}{4}\Bigr)\sqrt{\widetilde\rho\eps}\,\bI_{d_x}
=-\sqrt{\widetilde\rho\eps}\,\bI_{d_x}.
\end{align*}
\end{proof}

\section{The Proof of Section \ref{sec:lower}}

\subsection{Proof of Lemma \ref{lem:pro-psi}}

\begin{proof}
% Choose the target stationarity $\eps$ small enough that
% \begin{equation}\label{eq:lower-smallness}
% c_\eps\le L/8,\qquad 2G\eps\le L \sigma/ 8.
% \end{equation}
% Both conditions are of the form $\eps\le cL^2/(\rho\kappa^{3/2})$ with a universal positive constant $c$.
Recall that $Q(\bx)= \alpha \bar h_T(\bx/\sigma)+ (L/4)\norm{\bx}^2$. Then, by Lemma \ref{lem:carmon}(ii), the gradient of function $\alpha \bar h_T(\cdot/\sigma)$ is $c_\eps = \alpha\ell_1/\sigma^2$-Lipschitz continuous. Thus, the matrix $\nabla^2Q(\,\cdot\,)$ lies between $(L/2-c_\eps)\bI_T$ and $(L/2 +c_\eps)\bI_T$, and the function $Q(\,\cdot\,)$ is strongly convex. Let $\bx(\by)$ be the unique solution of $\nabla Q(\bx)=\beta \by$. For $\psi(\by) = Q^*(\beta\by)$, the differentiation of the conjugate function yields
\begin{equation}\label{eq:conjugate-derivatives}
\nabla\psi(\by)=\beta \bx(\by),\qquad
\nabla^2\psi(\by)=\beta^2\big(\nabla^2Q(\bx(\by))\big)^{-1}.
\end{equation}
For $c_\eps \le L/8$, the choice of $\beta^2=\mu(L/2+c_\eps)$ gives
\[
\mu \bI_T\preceq\nabla^2\psi(\by)
\preceq\mu\frac{\frac{L}{2}+c_\eps}{\frac{L}{2}-c_\eps}\bI_T
\preceq\frac{5}{3}\mu \bI_T.
\]
The $(L/2-c_\eps)$-strong convexity of the function $Q(\,\cdot\,)$ implies that its gradient is strongly monotone, and thus we have
\[
(\frac{L}{2} - c_\eps) \| \bx(\by) - \bx(\by') \| \le \| \nabla Q(\bx(\by)) - \nabla Q(\bx(\by')) \|.
\]
Equivalently, we have
\[
\| \bx(\by) - \bx(\by') \| \le \frac{\beta}{\frac{L}{2} - c_\eps} \| \by - \by' \|.
\]
Then, we can apply the inverse-matrix identity $\bA^{-1} - \bB^{-1} = \bA^{-1} (\bB- \bA) \bB^{-1}$ on
the Hessian matrix $\nabla^2\psi(\by)=\beta^2\big(\nabla^2Q(\bx(\by))\big)^{-1}$ to obtain that
% and  and $\norm{x(y)-x(y')}\le\beta\norm{y-y'}/(\lambda-c_\eps)$ imply
\begin{align}
\norm{\nabla^2\psi(\by)-\nabla^2\psi(\by')} \le\frac{\beta^3\rho_h}{(\frac{L}{2}-c_\eps)^3}\norm{\by-\by'}.
% \\
% &\le\frac{(5/8)^{3/2}}{(3/8)^3}\theta\rho\norm{y-y'}
% <\rho\norm{y-y'}.
\end{align}
% The numerical ratio is less than $9.4$, and $\theta=1/10$. 
% Only the $y$ block of the Hessian of $f$ varies with its argument. Furthermore,
% \[
% \norm{\nabla^2f}
% \le\lambda+\beta+\tfrac53\mu
% \le L\left(\tfrac12+\sqrt{\frac{5}{8\kappa}}+\frac{5}{3\kappa}\right)
% \le L
% \quad\text{for }\kappa\ge16.
% \]
% Thus $f$ is globally $L$-smooth, has a globally $\rho$-Lipschitz Hessian, and is $\mu$-strongly concave in $y$. Finally,
% \[
% \max_yf(x,y)=-\tfrac\lambda2\norm x^2+\psi^*(\beta x)=h(x),
% \]
% so \eqref{eq:scaled-gap-and-grad} gives the required initial primal gap.
\end{proof}

\subsection{Proof of Lemma \ref{lem:zero-chain}}

\begin{proof}
Suppose $\supp(\by) \subseteq \{1, \ldots, k\}$, equvalently, for any $j \ge k+1$, we have $y_j = 0$.
From the first-order optimality condition of the conjugate function $\nabla Q(\bx(\by)) = \beta \by$, the coordinates of $\bx(\by)$ satisfy
\[
x_j(\by) = \frac{2(\beta y_j - \frac{\alpha}{\sigma} \partial_j \bar h_T(\frac{\bx(\by)}{\sigma}))}{L}, \qquad |x_j(\by)| < \frac{2G\alpha}{L \sigma} \le \frac{\sigma}{2} \quad (\forall j \ge k+1),
\]
where the last inequality uses the condition $ G \alpha \le L\sigma^2/4$.
Let $\bx_{>i} \triangleq (x_{i+1}, x_{i+2}, \ldots, x_T)^\top \in \R^{T-i}$ denote the subvector comprising coordinates strictly greater than $i$.
According to Lemma \ref{lem:carmon}(iv), in a neighborhood of $\bx(\by)$, the function $Q$ decomposes into a function depending only on the first $k+1$ coordinates plus a decoupled quadratic term that is the squared norm of the other coordinates. Specifically, 
\[
Q(\bx) = Q_{k+1}(\bx_{\le k+1}) + \frac{L}{4} \| \bx_{>k+1} \|^2.
\]
We know from the envelope theorem that $\bx(\by)$ is continuous in $\by$. Thus, for every $\by'$ sufficiently close to $\by$, the maximizer $\bx(\by')$ remains in the same neighborhood of $\bx(\by)$, where the above decomposition holds. We can explicitly calculate the conjugate maximization over each tail coordinate, which is 
\[
\max_{x_j \in \R}\left\{\beta y_j'x_j-\frac{L}{4}x_j^2\right\}
=\frac{\beta^2}{L}(y_j')^2
\qquad (j>k+1).
\]
Summing over the tail coordinates therefore yields the local conjugate decomposition
\[
\psi(\by') = \psi_{k+1}(\by'_{\le k+1}) + \frac{\beta^2}{L}\norm{\by'_{>k+1}}^2.
\]
Differentiating this equation immediately yields the required zero-chain property in Eq. (\ref{eq:oracle-support}).
% \begin{equation}\label{eq}
% \supp \nabla\psi(\vy) \subseteq {1, \ldots, k+1}, \qquad
% \supp \left( \nabla^2\psi(\vy) - \frac{\beta^2}{\lambda} \mathbf{I} \right) \subseteq {1, \ldots, k+1}^2.
% \end{equation}
% [Add some explanations with fomulars and how the continuity of x(y) is used]
% By the envelope theorem (or continuity of $\vx(\vy)$), this primal decomposition yields the local conjugate dual decomposition:
% \begin{equation}\label{eq:conjugate-zero-chain}
% \psi(\vy') = \psi_{k+1}(\vy'_{\le k+1}) + \frac{\beta^2}{2\lambda}\norm{\vy'_{>k+1}}^2,
% \end{equation}
% for nearby $\vy'$. Differentiating \eqref{eq:conjugate-zero-chain} at $\vy$ immediately yields the required support properties:
% \begin{equation}\label{eq:oracle-support}
% \supp \nabla\psi(\vy) \subseteq \{1, \ldots, k+1\}, \qquad
% \supp \left( \nabla^2\psi(\vy) - \frac{\beta^2}{\lambda} \mathbf{I} \right) \subseteq \{1, \ldots, k+1\}^2.
% \end{equation}
\end{proof}

\subsection{Proof of Lemma \ref{lem:resisting}}

\begin{proof}
For any deterministic second-order algorithm as per Definition \ref{dfn:so-alg-class}, the $t$-th query point $(\bx^{(t)}, \by^{(t)})$ is a deterministic function of prior oracle replies.
Define $q(\by) \triangleq \psi(\by)- \beta^2\norm\by^2/L$ as the second-order zero-chain implied by Lemma \ref{lem:zero-chain}. Then, we can equivalently write our hard instance as
\[
f_\bU(\bx,\by)=-\frac{L}{4} \norm \bx^2+\beta\ip \bx \by-\frac{\beta^2}{L} \norm \by^2-q(\bU^\top \by).
\]
All dependence on the unknown rotation is in the function $q(\,\cdot\,)$ in the last term. By Lemma \ref{lem:zero-chain}, when coordinates $i,i+1,\ldots, T$ are zero, both the gradient and Hessians of the function $q$ involve only the previous coordinates $1,\ldots,i$. Therefore, we can construct $\bU$ adversarially as follows:
\begin{enumerate}
    \item Given the explored directions $\bu_1, \ldots, \bu_{t-1}$, we select $\bu_t$ orthogonal to $\bu_1, \ldots, \bu_{t-1}$ as well as all query vectors $\{\bx^{(s)}, \by^{(s)}\}_{s=1}^t$. This is always possible for $t < T$ under the assumed dimension.
    \item We complete the remaining basis vectors $\bu_{t+1}, \ldots, \bu_T$ to maintain orthogonality with all queries up to step $t$.    
\end{enumerate}
Under this construction, the $t$-th query has zero projections onto $\bu_t, \bu_{t+1}, \ldots, \bu_T$, so the second-order oracle response reveals information along $\bu_1, \ldots, \bu_t$ at most. The returned gradients and Hessians are entirely independent of the future vectors $\bu_{t+1}, \dots, \bu_T$, while earlier oracle responses remain unaffected because they rely exclusively on $\bu_1, \ldots, \bu_{t-1}$. We repeat the above procedure, and finally we can guarantee that $\langle \bx^{(T)}, \bu_T \rangle = 0$ by induction.

% Consider a deterministic algorithm. At query $t$, its point $(x^t,y^t)$ is determined by previous replies. Having fixed $u_1,\ldots,u_{t-1}$, choose $u_t$ orthogonal to these directions and all query vectors $x^1,y^1,\ldots,x^t,y^t$. This is possible in the chosen dimension for $t<T$. Complete the future columns orthogonally to all queries made so far. The current query then has zero coordinates along $u_t,u_{t+1},\ldots$, and the oracle reply can involve $u_1,\ldots,u_t$ only. Its value, gradient, and Hessian are independent of the later completion. Previous replies remain unchanged because they depended only on the earlier chosen directions.
% If the algorithm stops after $q<T$ queries and outputs $\hat x$, complete the columns $u_{q+1},\ldots,u_T$ orthogonally also to $\hat x$. There are enough dimensions: the span of the $2q$ query vectors, $q$ fixed directions, and the output has dimension at most $3q+1$, leaving at least $T-q$ dimensions. The completed instance is consistent with the entire transcript, and $\ip{u_T}{\hat x}=0$. 
\end{proof}

\subsection{Proof of Theorem \ref{thm:lower}}

\begin{proof}
%Use the definitions of our hard instance in \eqref{eq:hard-f}--\eqref{eq:lower-construction}, 
We verify that our hard instance in \eqref{eq:rotated-instance} lies in our function class. First, the Hessian smoothness in Assumption \ref{ass:smooth} holds due to the Hessian Lipschitzness in Lemma \ref{lem:pro-psi} and our parameter setting that
\[
\norm{\nabla^2\psi(\by)-\nabla^2\psi(\by')} \le\frac{\beta^3\rho_h}{(\frac{L}{2}-c_\eps)^3}\norm{\by-\by'}\ \le\frac{(\frac{5}{8})^{\frac{3}{2}}}{10 \cdot (\frac{3}{8})^3} \norm{\by-\by'} <\rho\norm{\by-\by'}.
\]
And the gradient smoothness in Assumption \ref{ass:smooth} is also satisfied, because for $\kappa \ge 16$,  we have
\[
\norm{\nabla^2f}
\le \frac{L}{2}+\beta+\frac{5}{3}\mu
\le L\left(\frac{1}{2}+\sqrt{\frac{5}{8\kappa}}+\frac{5}{3\kappa}\right)
\le L.
\]
Second, the strong concavity Assumption \ref{ass:concave} holds due to the lower curvature lower bound of $\bU^\top \by$ in Lemma \ref{lem:pro-psi} and the curvature in $\by_\perp$ is $2\beta^2/L \in[\mu,5\mu/4]$. Third, we examine the final Assumption \ref{ass:gap}. Since $\Phi_\bU(\bx) = \alpha \bar h_T(\bU^\top\bx/\sigma)$, Lemma \ref{lem:carmon}(i) indicates that
\begin{equation}\label{eq:scaled-gap-and-grad}
\Phi_\bU(\bzero)-\inf_{\bx \in \R^d} \Phi_\bU(\bx)\le12T\alpha.
% \qquad
% |x_T|<\sigma\ \Longrightarrow\ \norm{\nabla h(x)}>2\eps.
\end{equation}
Hence, $\Phi_\bU$ has initial gap at most $\Delta$ if we take
\begin{equation}\label{eq:chain-length}
T=
\left \lfloor \frac{\Delta}{12 \alpha}  \right \rfloor = 
\left\lfloor\frac{\Delta\sqrt{\rho_h}}{12\cdot 2^{\frac{3}{2}}\sqrt{\ell_2}\,\eps^{\frac{3}{2}}}\right\rfloor.
\end{equation}
Since $\alpha/\sigma = 2 \eps$ under the parameter setting in \eqref{eq:scaling}, we can combine Lemma \ref{lem:carmon}(v) and Lemma \ref{lem:resisting} to conclude that
\[
\norm{\nabla\Phi_\bU(\bx^{(t)})}
= \frac{\alpha}{\sigma} \norm{\bU \nabla \bar h_T({\bU^\top\bx^{(t)}}/{\sigma})}
= \frac{\alpha}{\sigma} \norm{\nabla \bar h_T({\bU^\top\bx^{(t)}}/{\sigma})}>2\eps.
\]
for any $t <T$. Substitution of $T$ in \eqref{eq:chain-length} proves Theorem~\ref{thm:lower} after choosing the universal constants.
% $\| \nabla \Phi_\mU(\vx^{(t)}) \| > 2 \eps$ for $t < T$. 
% % Furthermore,
% % Thus $f$ is globally $L$-smooth, has a globally $\rho$-Lipschitz Hessian, and is $\mu$-strongly concave in $y$.
% Then $\alpha/\sigma=2\eps$, $\norm{\nabla h}_\infty\le2G\eps$, $\norm{\nabla^2h}\le c_\eps$, and the Hessian of $h$ is $\rho_h$-Lipschitz. Further,
% Take
%  and the chain length has the required dependence.

% Consequently,
% This explicitly handles an arbitrary output; it does not assume that the output lies in the span of previously revealed directions. Therefore fewer than $T$ queries cannot guarantee the required stationarity. 
\end{proof}

% \begin{proof}
% Define $\eta=\lambda\sigma/(4\beta)$. Suppose $|y_j|<\eta$ for every $j\ge i$. The equation for $\vx(\vy)$ pf $\nabla Q(\vx(\vy)) = \beta \vy$ gives, coordinatewisely,
% \[
% x_j(y)=\frac{\beta y_j-\partial_jh(x(y))}{\lambda},\qquad
% |x_j(y)|<\frac{\beta\eta+2G\eps}{\lambda}\le\frac\sigma2
% \quad(j\ge i).
% \]
% The scaled local zero-chain property of $h$ now applies. 
% Let $\vx_{>i} \triangleq (x_{i+1}, x_{i+2}, \ldots, x_d)^\top \in \mathbb{R}^{d-i}$ denote the subvector comprising coordinates strictly greater than $i$. In a neighborhood of $x(y)$, the function $Q$ is a function of the first $i$ coordinates plus $\lambda\norm{x_{>i}}^2/2$. By continuity of $x(y)$, the same decomposition remains valid at the maximizers for nearby $y'$. Thus
% \begin{equation}\label{eq:conjugate-zero-chain}
% \psi(y')=\psi_i(y'_{\le i})+\frac{\beta^2}{2\lambda}\norm{y'_{>i}}^2
% \end{equation}
% locally, for a function $\psi_i$ involving only the first $i$ coordinates. In particular, if $\supp y\subseteq\{1,\ldots,k\}$, then
% \begin{equation}\label{eq:oracle-support}
% \supp\nabla\psi(y)\subseteq\{1,\ldots,k+1\},\qquad
% \nabla^2\psi(y)-\frac{\beta^2}{\lambda}I
% \text{ is supported on }\{1,\ldots,k+1\}^2.
% \end{equation}
% \end{proof}

\section{Experimental Details}\label{app:experiments}
All runs use double precision and a single thread on one AMD EPYC 9575F processor (64 cores, 566\,GB of RAM).
The software is Python~3.14, NumPy~2.5, and SciPy~1.18 with reference BLAS/LAPACK.

\subsection{Problem Parameters}
We choose  the regularizer as
\[
r(\bx)=\frac{\lambda_0}{2}\norm{\bx}^2+\lambda_1\sum_{j=1}^{d}\frac{x_j^2}{1+x_j^2},
\]
where $\lambda_0=10^{-4}$ and $\lambda_1=10^{-2}$.
Let $\bA$ be the data matrix with rows $\ba_i^\top$, and set $c=\norm{\bA}^2/(4n)$.
For a prescribed condition number $\kappa$, we choose
\[
\mu=\frac{3c+\lambda_0+2\lambda_1}{\kappa-1}
\qquad\text{and}\qquad
\nu=c+\mu.
\]
The Hessian of the logistic loss is bounded above by $c\bI_d$, so $f(\bx,\cdot)$ is $\mu$-strongly concave.
Assumption~\ref{ass:smooth} holds with the parameters 
\[
L=3c+\lambda_0+2\lambda_1+\mu
\qquad\text{and}\qquad
\rho=\frac{\sqrt2}{3\sqrt3\,n}\sum_{i=1}^{n}\norm{\ba_i}^3+12\lambda_1.
\]
At the initial points used in the experiments, the computed values of $\lambda_{\min}(\nabla^2\Phi(\bx_0))$ range from $-5\cdot10^{-3}$ to $-3\cdot10^{-3}$, confirming that $\Phi(\cdot)$ is nonconvex in every setting.

\subsection{Parameter Tuning}

We tune each method on a fixed stratified subset of 2{,}000 ``a9a'' training examples.
Each parameter choice is tested at $\kappa\in\{10,100\}$ from two initial points, giving four tuning runs.
Each tuning run uses $\eps=10^{-3}$ and a time limit of 300\,s.
Table~\ref{tab:experiment-tuning} lists the tuned parameters, their initial search grids, and the selected multipliers.
For ECRN, $H$ is the cubic coefficient and $R$ is the continuation radius; $c_R=+\infty$ means $m=1$ in the NewtonTrack subroutine.
For MCN, $M$ is the cubic coefficient and $\bar K_t$ is the inner AGD iteration count prescribed by \citet{luo2022}.
For GDA, the dual step size is fixed at $\eta_y=1/L$.
IAPUN tunes $L_2$ as its Hessian-Lipschitz parameter.
\begin{table}[htbp]
\centering
\caption{Initial parameter grids and selected multipliers.}
\label{tab:experiment-tuning}
\setlength{\tabcolsep}{5pt}
\renewcommand{\arraystretch}{1.4}
\begin{tabular}{@{}llcc@{}}
\toprule
\textbf{Method} & \textbf{Parameter} & \textbf{Initial Grid} & \textbf{Selected Multiplier}\\
\midrule
ECRN & $H=c_H\rho(4+\kappa)^{\frac{3}{2}}$ & $\{1,10^{-1},\ldots,10^{-6}\}$ & $c_H=10^{-4}$\\
     & $R=c_R\frac{\mu}{48\rho\sqrt\kappa}$ & $\{1,30,10^3,+\infty\}$ & $c_R=+\infty$\\
MCN  & $M=c_M\,4\sqrt2\,\kappa^3\rho$ & $\{1,10^{-1},\ldots,10^{-5}\}$ & $c_M=10^{-6}$\\
     & $K_t=\lceil c_K\bar K_t\rceil$ & $\{1,10^{-1},\ldots,10^{-3}\}$ & $c_K=10^{-2}$\\
GDA  & $\eta_x=\frac{c_\eta}{16(\kappa+1)^2L}$ & $\{1,10,\ldots,10^4\}$ & $c_\eta=10^4$\\
IAPUN & $L_2=c_L\rho(1+\kappa)^3$ & $\{1,10^{-1},\ldots,10^{-5}\}$ & $c_L=10^{-1}$\\
\bottomrule
\end{tabular}
\end{table}

During tuning, ECRN, MCN, and IAPUN use their original stopping rules.
We stop GDA at the first iterate satisfying
$
\norm{\nabla_xf(\bx,\by)}+\kappa\norm{\nabla_yf(\bx,\by)}\le\eps,
$
whose left-hand side is a upper bound on $\norm{\nabla\Phi(\bx)}$.
A parameter choice is eligible only if all four runs stop within the time limit and attain a certified gradient norm at most $\eps$.
Among eligible choices, we select the one with the smallest running time in average.
The selected multipliers in Table~\ref{tab:experiment-tuning} are used for all reported experiments.
For $c_M$, $c_L$, and $c_\eta$, we extend the grid by two powers of ten whenever the best choice lies at the searched endpoint.
The extended searches tested $c_M\in\{10^{-6},10^{-7}\}$ and $c_\eta\in\{10^5,10^6\}$.

\subsection{Evaluation Protocol}
For each dataset and condition number, all methods use the same five initial points drawn from $\mathcal N(\bzero,\bI_d)$, with $\by_0=\bzero$.
We evaluate stationarity using
\[
\nabla\Phi(\bx)=\nabla_xf\bigl(\bx,\by^*(\bx)\bigr).
\]
The maximization $\by^*(\bx)=\max_{\by\in\R^{d_y}}f(\bx,\by)$ is solved by damped Newton iterations so that the reported gradient has an error of at most $10^{-9}$.
These diagnostic evaluations are excluded from the reported oracle counts and running times.
Each run stops when its certified gradient norm is at most $10^{-4}$, or when both its oracle count and running time exceed the corresponding maximum budgets in Figures~\ref{fig:experiments} and~\ref{fig:exp-k100}.

% \subsection{Construction of the Curves}

% For each method and each budget, we first compute the smallest certified gradient norm reached in each of the five runs within that budget.
% We then plot the arithmetic mean of these five minima in Figures~\ref{fig:experiments} and~\ref{fig:exp-k100}.
% The budget is measured by oracle calls or running time; the number of completed iterations may differ across runs.

% Each horizontal axis starts at a rounded positive budget to omit the initial nearly flat portion of the curves.
% At every omitted budget, all curves remain at least $90\%$ of their initial values, corresponding to a decrease of at most $10\%$.
% All methods within a panel share the same starting budget.

\end{document}